\documentclass[leqno,a4paper]{article}
\usepackage{amsmath,amsthm,amssymb}
\usepackage[utf8]{inputenc}
\usepackage[T1]{fontenc}

\usepackage{enumerate}
\usepackage{bbm}
\usepackage{todonotes}
\usepackage{mathrsfs}
\usepackage{mathabx}
\usepackage{hyperref}
\usepackage{nicefrac}
\usepackage{tikz}
\usetikzlibrary{matrix}

\usepackage{url}
\makeatletter
\g@addto@macro{\UrlBreaks}{\UrlOrds}
\makeatother

\usepackage[sort,numbers]{natbib}

\providecommand{\noopsort}[1]{} 

\def\qed{\unskip\quad \hbox{\vrule\vbox
to 6pt {\hrule width 4pt\vfill\hrule}\vrule} }

\newtheorem{Th}{Theorem}[section]
\newtheorem{Prop}[Th]{Proposition}
\newtheorem{Lemma}[Th]{Lemma}

\theoremstyle{definition}
\newtheorem{Remark}[Th]{Remark}
\newtheorem{Def}{Definition}[section]
\newtheorem{Cor}[Th]{Corollary}

\newcommand{\beq}{\begin{equation}}
\newcommand{\eeq}{\end{equation}}

\def\scalar(#1,#2){(#1\mid#2)}

\newcommand{\raz}{\mathbbm{1}}

\newcommand{\cs}{{\cal S}}
\newcommand{\ca}{{\cal A}}
\newcommand{\cb}{{\cal B}}
\newcommand{\cc}{{\cal C}}
\newcommand{\cd}{{\cal D}}

\newcommand{\cj}{{\cal J}}
\newcommand{\cp}{{\cal P}}

\newcommand{\ch}{{\cal H}}
\newcommand{\cm}{{\cal M}}

\newcommand{\cw}{{\cal W}}
\newcommand{\xbm}{(X,{\cal B},\mu)}
\newcommand{\zdr}{(Z,{\cal D},\rho)}
\newcommand{\ycn}{(Y,{\cal C},\nu)}
\newcommand{\ot}{\otimes}
\newcommand{\ov}{\overline}
\newcommand{\la}{\lambda}

\newcommand{\bs}{\mathbb{S}}

\newcommand{\Q}{\mathbb{Q}}
\newcommand{\R}{{\mathbb{R}}}
\newcommand{\T}{{\mathbb{T}}}
\newcommand{\C}{{\mathbb{C}}}
\newcommand{\Z}{{\mathbb{Z}}}
\newcommand{\N}{{\mathbb{N}}}

\newcommand{\vep}{\varepsilon}
\newcommand{\va}{\varphi}

\newcommand{\mob}{\boldsymbol{\mu}}
\newcommand{\lio}{\boldsymbol{\lambda}}

\newcommand{\tfs}{T_{\va,\cs}}
\newcommand{\tf}{T_{\va}}
\newcommand{\bfu}{\boldsymbol{u}}

\title{M\"obius disjointness along ergodic sequences for uniquely ergodic actions}
\author{Joanna Ku\l aga-Przymus* \and Mariusz Lema\'nczyk\thanks{Research supported by Narodowe Centrum Nauki grant  UMO-2014/15/B/ST1/03736.}}

\begin{document}
\bibliographystyle{siam}

\maketitle

\begin{abstract} We show that there are an irrational rotation $Tx=x+\alpha$ on the circle $\T$ and a continuous $\va\colon\T\to\R$ such that
for each (continuous) uniquely ergodic flow $\cs=(S_t)_{t\in\R}$ acting on a compact metric space $Y$, the automorphism $\tfs$ acting on $(X\times Y,\mu\ot\nu)$ by the formula
$\tfs(x,y)=(Tx,S_{\va(x)}(y))$, where $\mu$ stands for Lebesgue measure on $\T$ and $\nu$ denotes the unique $\cs$-invariant measure, has the property of asymptotically orthogonal powers.  This gives a class of relatively weakly mixing extensions of irrational rotations for which Sarnak's conjecture on M\"obius disjointness holds for all uniquely ergodic models of $\tfs$. Moreover, we obtain a class of ``random''  ergodic sequences $(c_n)\subset\Z$ such that if $\mob$ denotes the M\"obius function, then
$$
\lim_{N\to\infty}\frac1N\sum_{n\leq N}g(S_{c_n}y)\mob(n)=0$$
for all (continuous) uniquely ergodic flows $\cs$, all $g\in C(Y)$ and $y\in Y$.

\end{abstract}

\tableofcontents

\section{Introduction} 
In 2010, P.\ Sarnak \cite{Sa} stated the following conjecture: {\em Given a zero (topological) entropy homeomorphism $S$ of a compact metric space $Y$, we have} 
\beq\label{mob00}
\lim_{N\to\infty}\frac1N\sum_{n\leq N}g(S^ny)\mob(n)=0
\eeq 
{\em  for each $g\in C(Y)$ and $y\in Y$}.
Here $\mob\colon\N\to\C$ denotes the M\"obius function: $\mob(1)=1$, $\mob(p_1\ldots p_k)=(-1)^k$ for any distinct prime numbers $p_1,\ldots, p_k$ and $\mob(n)=0$ otherwise. Given a dynamical system $(Y,S)$, we will say that $(Y,S)$, or simply $S$, is {\em M\"obius disjoint} when~\eqref{mob00} holds for all $g\in C(Y)$ and $y\in Y$.
M\"obius disjointness has been proved in numerous cases of zero entropy dynamical systems, see e.g.\ \cite{Ab-Ka-Le,Bo-Sa-Zi,Do-Ka,Dr,Fe-Ku-Le-Ma,Gr-Ta,Karagulyan,Pe}.

One of motivations of the present paper is to consider a ``randomized'' version of M\"obius disjointness in which we are given another compact metric space $X$, a homeomorphism $T\colon X\to X$  and a family $(S_x)_{x\in X}$ of homeomorphisms of $Y$ (the map $(x,y)\mapsto S_xy$ is assumed to be continuous) and  we aim at proving that
\beq\label{mobran1}
\lim_{N\to\infty}\frac1N\sum_{n\leq N}g(S_x^{(n)}y)\mob(n)=0\eeq
for each $x\in X$, $g\in C(Y)$ and $y\in Y$. Here $S^{(n)}_x=S_{T^{n-1}x}\circ\ldots\circ S_{Tx}\circ S_x$ for $n\geq 0$ ($S^{(0)}_x=Id$), so when $X$ is the one-point space, clearly, \eqref{mobran1} is equivalent to~\eqref{mob00}. An instance of such ``randomized'' version arises when we choose
a continuous function $\va\colon X\to \R$ and consider $\cs=(S_t)_{t\in\R}$  a continuous flow on $Y$. Now, \eqref{mobran1} takes the form
\beq\label{mobran1prime}
\lim_{N\to\infty}\frac1N\sum_{n\leq N}g(S_{\va^{(n)}(x)}y)\mob(n)=0\eeq
for each $x\in X$, $g\in C(Y)$ and $y\in Y$ ($\va^{(n)}(x)=\sum_{j=0}^{n-1}\va(T^jx)$).
We can ask now if we can find ``parameters'' $(T,\va)$ so that \eqref{mobran1prime} holds for ``all'' $\cs$.
Somewhat surprisingly, the answer to such a question will be positive, whenever  by ``all'' we mean  the class of continuous uniquely ergodic flows, see Theorem~A below.

To explain better our approach, suppose that $S$ acting on $Y$ is M\"obius disjoint and, additionally, it is uniquely ergodic (such are most of examples from the aforementioned papers). Assume that $S'$ is another uniquely ergodic homeomorphism acting on $Y'$. If $S$ and $S'$ are topologically conjugate then also $S'$ is M\"obius disjoint. But both dynamical systems $(Y,S)$ and $(Y',S')$ have unique invariant measures, $\nu$ and $\nu'$, respectively, and it is natural to ask whether already a metric isomorphism of measure-theoretic dynamical systems $(Y,\cb(Y),\nu,S)$ and $(Y',\cb(Y'),\nu',S')$ implies that $S'$ is M\"obius disjoint ($\cb(Y)$ and $\cb(Y')$ stand  for the $\sigma$-algebras of Borel sets on $Y$ and $Y'$, respectively).  This   ergodic approach to study convergence~\eqref{mob00}, in fact, allows us to ask under which measure-theoretic properties an ergodic automorphism $W$ of a probability standard Borel space $\zdr$ has the property that the M\"obius disjointness holds for each homeomorphism $S$ of a compact metric space $Y$ which is a uniquely ergodic model of $(Z,\cd,\rho,W)$. The latter means that the measure-theoretic dynamical systems $(Z,\cd,\rho,W)$ and $(Y,\cb(Y),\nu,S)$, where $\nu$ is the unique $S$-invariant probability measure,  are measure-theoretically isomorphic (we recall that by the Jewett-Krieger theorem, each ergodic $W$ has a uniquely ergodic model).  As noticed in \cite{Bo-Sa-Zi},  this question has the  positive answer if the different prime powers of $W$ are disjoint in Furstenberg's sense. A generalization of the disjointness of powers proposed in \cite{Ab-Le-Ru} (called AOP, see  Subsection~\ref{SectAOP}), allows one  to cover the case when such powers are not disjoint and, in the extremal case, can even be isomorphic. 

Each AOP automorphism has the property that all its uniquely ergodic models are M\"obius disjoint.
For example, the M\"obius disjointness is known to hold for each irrational rotation (e.g.\ \cite{Da}) but by \cite{Ab-Le-Ru} it follows that irrational rotations  enjoy the AOP property. Hence, the M\"obius disjointness holds  in each ergodic  model of an irrational rotation, e.g.\ in Sturmian models \cite{Arnoux} or topologically mixing models \cite{Leh}. Similarly, M\"obius disjointness for zero entropy affine automorphisms has been proved in \cite{Li-Sa}, while
the main result of \cite{Ab-Le-Ru} asserts that all quasi-discrete spectrum automorphisms \cite{Ab}  enjoy the AOP property; for M\"obius disjointness itself of zero entropy affine automorphisms (which are examples of automorphisms with quasi-discrete spectrum), see \cite{Li-Sa}. It follows that M\"obius disjointness holds in all uniquely ergodic models of zero entropy affine automorphisms.
Recently, in \cite{Fl-Fr-Ku-Le}, the AOP property has been proved for all unipotent diffeomorphisms of compact nil-manifolds. It can also be shown that some examples from \cite{Ku-Le} and \cite{Wang} enjoy the AOP property. The classes of automorphisms we listed in this paragraph are examples of distal automorphisms \cite{Fu}, in fact, these are distal extensions of some totally ergodic rotations.
From that point of view, in the present paper, we will deal with the notion complementary to distality. Namely, we will mainly consider (uniquely ergodic) relatively weakly mixing extensions \cite{Fu} of irrational rotations, and the paper is aimed at a study of the AOP property for them.
Once the AOP property is established, such relatively weakly mixing extensions of rotations  will constitute a new class of dynamical systems for which the M\"obius disjointness holds  in all their uniquely ergodic models.

Let us describe more precisely the class of systems we intend to study. They are of the following form:
\beq\label{RokExt}\begin{array}{c} \tfs:(X\times Y,\cb\ot\cc,\mu\ot\nu)\to(X\times Y,\cb\ot\cc,\mu\ot\nu),\\ \tfs(x,y)=(Tx,S_{\va(x)}(y)),\end{array}\eeq
where $X=\T=[0,1)$ (mod~1), $Tx=x+\alpha$ is an irrational rotation, $\va\colon X\to\R$ is measurable  and $\cs=(S_t)_{t\in\R}$ is an ergodic flow acting on a
probability standard Borel space $\ycn$. Under a weak assumption on $\va$, the skew products~\eqref{RokExt} are ergodic \cite{Le-Pa1}. Furthermore, if additionally $\va$ is continuous  then $\tfs$ is a homeomorphism which is uniquely ergodic for each  uniquely ergodic flow $\cs$. Under some further assumptions on $\va$ and $\cs$ such extensions are relatively weakly mixing extensions of $T$ \cite{Le-Pa1,Le-Pa2}. For example, these assumptions are
satisfied when $\tf$ is ergodic and $\cs$ is weakly mixing. Recall that $\tf$ is defined on $X\times\R$ by the formula $\tf(x,t)=(Tx,\va(x)+t)$ and it preserves $\mu\ot\la_{\R}$;  we assume ergodicity of $\tf$ with respect to this infinite measure. Furthermore, since the cocycles $\va$ we use are recurrent, the entropy of $\tfs$ will also be zero regardless the entropy of $\cs$ itself \cite{Da} (recall also that the AOP property implies zero entropy \cite{Ab-Le-Ru}). Finally, we note that given $g\in C(Y)$ if we set $F(x,y)=g(y)$ then
$$
\frac1N\sum_{n\leq N} F(\tfs^n(x,y))\mob(n)=\frac1N\sum_{n\leq N}g(S^{(n)}_{\va(x)}y)\mob(n)=\frac1N\sum_{n\leq N}g(S_{\va^{(n)}(x)}y)\mob(n),$$
where $\va^{(n)}(x)=\va(x)+\va(Tx)+\ldots+\va(T^{n-1}x)$, which explains why the M\"obius disjointness~\eqref{mob00} for $\tfs$ implies~\eqref{mobran1prime}.

One of the main results of the paper is the following.

\vspace{1ex}

{\bf Theorem~A.} {\em There are an irrational $\alpha\in\T$ and a measurable (even smooth) $\va\colon X\to\R$ such that $\tfs$ has the AOP property for each ergodic flow $\cs$.
In particular, there are an irrational rotation $T$ and a continuous $\va\colon\T\to\R$ such that ``randomized'' M\"obius disjointness~\eqref{mobran1} holds for all continuous, uniquely ergodic flows $\cs$, i.e.
$$
\lim_{N\to\infty}\frac1N\sum_{n\leq N}g(S_{\va^{(n)}(x)}y)\mob(n)=0$$
for all $x\in X$, $g\in C(Y)$ and $y\in Y$.}

\vspace{1ex}

\noindent (For the proof, see Theorem~\ref{jm}, Corollary~\ref{c:jm}, Proposition~\ref{wn:zalozenia} below.) We also consider the affine case: $\va(x)=x-\frac12$ and prove that $\tfs$ has the AOP property for each ergodic flow $\cs$ acting on $\ycn$ whose spectrum on $L^2_0\ycn$ is disjoint with $\Q$, see Theorem~\ref{t:aopAFF} below.

Assume now that $\va\colon X\to\R$ is continuous and the homeomorphisms $\tfs$ enjoy the AOP property for each (continuous) uniquely ergodic $\cs$.  It follows from Theorem~A that if we fix $x_0\in X$ and set \beq\label{cn}
c_n:=\va^{(n)}(x_0)=\va(x_0)+\va(Tx_0)+\ldots+\va(T^{n-1}x_0),\eeq then for each uniquely ergodic flow $\cs$ on a compact metric space $Y$,
$g\in C(Y)$ and $y\in Y$, we obtain
\beq\label{SarCon1}
\lim_{N\to\infty}\frac1N\sum_{n\leq N}g(S_{c_n}y)\mob(n)= 0.\eeq
Thus, $(c_n)$ is an example of a sequence along which the M\"obius disjointness holds for each uniquely ergodic flow $\cs$.
By taking
\beq\label{an}
a_n:=[c_n], \;n\geq1,\eeq
and applying the standard suspension construction (see Subsection~\ref{SecSusp}), we obtain
\beq\label{SarCon2}
\lim_{N\to\infty}\frac1N\sum_{n\leq N}b(R^{a_n}z)\mob(n)=0\eeq
for each uniquely ergodic homeomorphism $R$ of a compact metric space $Z$ ($b\in C(Z)$, $z\in Z$). In fact, we prove the $\bfu$-{\em disjointness}, i.e.\ we prove~\eqref{SarCon2} in which $\mob$ is replaced by a multiplicative function\footnote{The function $\mob$ is multiplicative, that is, $\mob(mn)=\mob(n)\mob(m)$ for each $m,n$ coprime.} or being more precise: for each multiplicative function $\bfu\colon\N\to\C$, $|\bfu|\leq1$, we have
\beq\label{SarCon3}
\lim_{N\to\infty}\frac1N\sum_{n\leq N}b(R^{a_n}z)\bfu(n)= 0\eeq
for each uniquely ergodic homeomorphism $R$ of a compact metric space $Z$, each $b\in C(Z)$, $\int_Z b\,d\rho=0$ ($\rho$ stands for the unique $R$-invariant measure) and $z\in Z$.

Notice however that the existence of a sequence $(a_n)\subset\Z$ for which the subsequence version~\eqref{SarCon2} of M\"obius disjointness holds for each uniquely ergodic homeomorphisms $R$ is not surprising. Indeed, each slowly increasing sequence $(a_n)$ of integers will do the same because the sequence $(b(R^{a_n}z))_{n\geq1}$ will be a bounded sequence which behaves like a constant sequence and~\eqref{SarCon2} will follow from the fact that
$\frac1N\sum_{n\leq N}\mob(n)\to0$.\footnote{We would like to thank N.\ Frantzikinakis and B.\ Weiss for some fruitful discussions on the subject.
N.\ Frantzikinakis noticed additionally that  the sequence $([n^c])$ with $0<c<1$, satisfies~\eqref{SarCon2} for each uniquely ergodic $R$.} On the other hand, the existence of $(a_n)$ for which~\eqref{SarCon3} holds for each multiplicative $\bfu\colon\N\to\C$, $|\bfu|\leq1$, and each uniquely ergodic $R$ does not seem to be clear as there are multiplicative functions  whose averages do not converge:  $n\mapsto n^{it}=e^{it\log n}$ for $t\neq0$, are examples of such. Moreover, whenever $\int_Zb\,d\rho=0$, the sequence $(b(R^{[a_n]}z))$ displays an orthogonality behaviour along different subsequences $(b(R^{[a_{pn}]}z)$ and $b(R^{[a_{qn}]}z)$ for different prime numbers $p,q$, namely, it will satisfy the assumptions of K\'atai-Bourgain-Sarnak-Ziegler criterion (see Proposition~\ref{kbszcrit} below). Finally, as proved recently in \cite{Ab-Ku-Le-Ru1}, the AOP property implies
a property similar to~\eqref{SarCon3} on a typical short interval. More precisely, it follows from \cite{Ab-Ku-Le-Ru1} that the following holds:

\vspace{1ex}

{\bf Corollary~B.} {\em There is a Poincar\'e sequence $(a_n)_{n\in\N}\subset\Z$~\footnote{A sequence $(a_n)$ is called Poincar\'e if for each ergodic automorphism $R$ of a  probability standard Borel space $\zdr$ and each $C\in{\cal D}$, $\rho(C)>0$, we have $\rho(R^{-a_n}C\cap C)>0$ for infinitely many $n$. Each ergodic sequence is Poincar\'e.

We provide quite explicit sequences satisfying the above (see Section~\ref{s:AOPaf}). For example, if $\alpha$ is irrational with bounded partial quotients and $\alpha,\beta,1$ are rationally independent then we can take $$a_n=\left[n\beta+\frac{n(n-1)}2\alpha-\frac n2-\sum_{j=1}^{n-1}[\beta+j\alpha]\right],\,n\geq1.$$}  such that for each multiplicative function
$\bfu\colon\N\to\C$, $|\bfu|\leq1$, we have
\beq\label{short1}
\frac1M\sum_{M\leq m<2M}\left|\frac1H\sum_{m\leq h<m+H}
b(R^{a_h}z)\bfu(h)\right|\to 0, \;H\to\infty,\;H/M\to0
\eeq
for each uniquely ergodic homeomorphism $R$ of a compact metric space $Z$,  $b\in C(Z)\cap L^2_0\zdr$  and $z\in Z$.}

\vspace{1ex}

\noindent

Proceeding as in \cite{Ab-Le-Ru}, it follows that
for each  degree $d>0$ polynomial $P\in\R[x]$ with the leading coefficient $\alpha_d$ irrational, we have
\beq\label{short2} \frac1M\sum_{M\leq m<2M}\left|\frac1H\sum_{m\leq h<m+H}
b(W^{a_h}z)e^{2\pi iP(a_h)} \bfu(h)\right|\to 0\eeq
when  $H\to\infty,\;H/M\to0$, for $W$ a uniquely ergodic homeomorphism such that $(Z,\cd,\rho,W)$ has no $e^{2\pi ik\alpha}\neq1$ as its eigenvalue, $\bfu$ as above and arbitrary $b\in C(Z)$ and $z\in Z$.\footnote{We only need to show  that the transformation $R'$ which is behind the sequence $e^{2\pi iP(n)}$, $n\geq0$,  is disjoint from $W$ as then $R:=W\times R'$ is also uniquely ergodic and we can use Corollary~B. As a matter of fact the transformation $R'$ is a several step affine extension of the irrational rotation by $\alpha:=d!\alpha_d$. Since $W$ is disjoint with this rotation, it is also disjoint with its compact group extension (which does not add new eigenvalues) \cite{Fu}.}

To illustrate~\eqref{short1}, consider  $R$ being the rotation on $Z=\Z/2\Z=\{0,1\}$, $Ri=i+1$, $b(i)=(-1)^i$, and $z=0$. The validity of~\eqref{short1} for these parameters yields the following:

\vspace{1ex}

{\bf Corollary~C.} {\em There is a Poincar\'e sequence $(a_n)\subset\Z$ such that for each multiplicative function
$\bfu\colon\N\to\C$, $|\bfu|\leq1$, we have
\beq\label{short3}
\frac1M\sum_{M\leq m<2M}\left|\frac1H\sum_{m\leq h<m+H}
(-1)^{a_h} \bfu(h)\right|\to 0, \;H\to\infty,\;H/M\to0.
\eeq}

\vspace{1ex}

\noindent Note that no assumption on the convergence of the averages of $\bfu$  is made. (If we assume that the averages of $\bfu$ are going to zero and $\bfu$ is real valued, then \eqref{short3} holds for any constant sequence $(a_n)$ \cite{Ma-Ra}.)

\section{Basic notions: cocycle, Mackey action, Rokhlin extension, joining}
\subsection{Mackey action associated to a cocycle} Let $\xbm$ be a probability standard Borel  space. Let ${\rm Aut}\xbm$ denote the group of (measure-preserving) automorphisms of $\xbm$.
Assume that $T\in {\rm Aut}\xbm$ is ergodic. Let $G$ be a locally compact second countable (lcsc) Abelian group and let  $\va\colon X\to G$ be measurable (we say that $\va$ is a \emph{cocycle}). By $\tf$ we denote the corresponding {\em group extension}:
$$
T_\va:(X\times G,\cb(X\times G),\mu\ot\la_G)\to (X\times G,\cb(X\times G),\mu\ot\la_G),
$$
$$
\tf(x,g)=(Tx,\va(x)+g);
$$
here $\la_G$ stands for a Haar measure of $G$; hence,$T_\va$ preserves the  measure $\mu\ot \la_G$ which is infinite if $G$ is not compact.
Note that $(T_\va)^k(x,g)=(T^kx,\va^{(k)}(x)+g)$, where
$$
\va^{(k)}(x)=\left\{\begin{array}{ccc}
\va(x)+\va(Tx)+\ldots+\va(T^{k-1}x) &\text{if}& k\geq1,\\
0&\text{if}&k=0,\\
-(\va(T^kx)+\ldots+f(T^{-1}x))&\text{if}&k<0.\end{array}\right.$$
Let $\tau=(\tau_g)_{g\in G}$ be the natural $G$-action on $(X\times G,\mu\ot\la_G)$:
\beq\label{act1}
\tau_g(x,g')=(x,g+g')\text{ for }(x,g')\in X\times G.\eeq
Then $\tau$ preserves the measure $\mu\ot\la_G$ and
for each $g\in G$, we have
\beq\label{act2}
\tf\circ \tau_g=\tau_g\circ \tf.\eeq
We say that $\va$ is {\em ergodic} if $\tf$ is ergodic. In general, $\va$ is not ergodic. For example, if
$$
\va(x)=\theta(Tx)-\theta(x)$$
for a measurable $\theta\colon X\to G$, i.e.\ when $\va$ is a {\em coboundary}, then clearly $\tf$ is not ergodic as every set \beq\label{sklerg}\{\tau_g(x,\theta(x)) : x\in X\},\;g\in G,\eeq is $\tf$-invariant. Fix a probability measure $\la$ equivalent to $\la_G$ and note that now $\tf$ and $\tau$ become non-singular actions on the  probability (standard Borel) space $(X\times G,\mu\ot\la)$. Let $\cj_\va=\cj(\tf)$ denote the $\sigma$-algebra of $\tf$-invariant sets and $(C_\va,\cj_\va,\kappa_\va)$ denote the corresponding probability (standard Borel)  quotient space, called the {\em space of ergodic components of} $\tf$. Since~\eqref{act2} holds, $\tau$ also acts on the space of ergodic components. This measurable and non-singular\footnote{Given a $G$-action $\cw=(W_g)_{g\in G}$ on a probability standard Borel space $\zdr$, we say that it is {\em measurable} if
the map $(z,g)\mapsto W_gz$ is measurable and it is {\em non-singular} if for each $g\in G$, the measure $\rho\circ W_g=(W_g)_\ast(\rho)$ given by $(W_g)_\ast(\rho)(A):=\rho(W_g^{-1}A)$ is equivalent to $\rho$. Implicitly, all actions under consideration are measurable.} $G$-action is called the {\em Mackey action} of $\tf$ (or of $\va$) and is denoted by $\cw(\va)$ or $\cw(\va,T)$ or even $\cw(\va,T,\mu)$ if not expliciting the parameters may lead to a confusion. The Mackey action is always ergodic. Note that the Mackey action of an ergodic cocycle is trivial (the action on the one point space), while it is the action of $G$ on itself (by translations) when $\va$ is a coboundary, cf.~\eqref{sklerg}.

\subsection{Essential values of a cocycle}
Let $T\in{\rm Aut}\xbm$ be ergodic and $\va\colon X\to G$ a cocycle with values in an lcsc Abelian  group $G$. Following \cite{Sch}, an element $g\in G$ is called an {\em essential value} of $\va$ if for every $C\in\cb$, $\mu(C)>0$, and every open $V\subset G$, $g\in V$,  there exists $N\in\Z$ such that
$
\mu(C \cap T^{-N}C \cap [\va^{(N)}\in  V]) > 0$.
By $E(\va)$ we denote the set of all essential values. In fact, it is a closed subgroup of $G$.

\begin{Prop}[\cite{Sch}]\label{p:essv} $\tf$ is ergodic if and only if $E(\va) = G$.\end{Prop}

We will later need the following fact from \cite{Sch}.

\begin{Prop}[\cite{Sch}]\label{p:disjess} Let $K\subset G$ be compact and $K\cap E(\va)=\emptyset$. Then there exists
$B\in\cb$, $\mu(B) > 0$, such that for each integer $m \geq1$, we have
$$
\mu(B \cap T^{-m}B \cap [\va^{(m)}\in  K]) = 0.$$
\end{Prop}

A cocycle $\va\colon X\to G$ is called {\em regular} if there exist a closed subgroup $H\subset G$, a cocycle $\psi\colon X\to H$ and $f\colon X\to G$ measurable such that
$\va=\psi+f\circ T-f$ and $\psi$ is ergodic as a cocycle taking values in $H$ (in fact, $H$ must be equal to $E(\va)$). In particular, when $H=\{0\}$, i.e.\ when $\va$ is a {\em coboundary}, then $\va$ is regular. Moreover, each cocycle $\va\colon X\to G$ for which $E(\va)$ is cocompact is regular.

A method to compute essential values is contained in the following.

\begin{Prop}[\cite{Le-Pa-Vo}] \label{RigEssVal} Assume that $T\in {\rm Aut}\xbm$ is ergodic and rigid, that is, $T^{q_n}\to Id$ (in $L^2\xbm$) for some increasing sequence $(q_n)\subset\N$. Let $\va\colon X\to G$ be a cocycle and suppose that the sequence  $(\va^{(q_n)})_\ast(\mu)$ of probability measures on $G$ weakly converges to a probability measure $P$ (on $G$). Then each point $g$ in the topological support of $P$ belongs to $E(\va)$.\end{Prop}

We will need to apply the above result when $\va,\psi\colon X\to G$ are given and we want to obtain some essential values  for the cocycle $\va+\psi$ out of $E(\va)$ and $E(\psi)$. The lemma below will be applied when the assumptions of Proposition~\ref{RigEssVal} are satisfied and $\va_n=\va^{(q_n)}$, $\psi_n=\psi^{(q_n)}$.

\begin{Lemma}\label{suma} Let $\xbm$ be a probability space, $G$ an lcsc Abelian group with an invariant metric $d$. Assume that $\va_n,\psi_n\colon X\to G$ are measurable and taking values in a compact set $C\subset G$. Moreover, assume that $\psi_n$ takes values in a finite set $F\subset C$, $n\geq1$.
Assume that $(\va_n)_\ast(\mu)\to\nu$, $(\psi_n)_\ast(\mu)\to\kappa$ and $(\va_n+\psi_n)_\ast(\mu)\to\rho$. Then for each $h_0\in{\rm supp}\,\kappa$ there exists $g_0\in{\rm supp}\,\nu$ such that $h_0+g_0\in{\rm supp}\,\rho$. Moreover,  for each $g_0\in{\rm supp}\,\nu$ there exists $h_0\in{\rm supp}\,\kappa$ such that $h_0+g_0\in{\rm supp}\,\rho$.
\end{Lemma}
\begin{proof}
(a) Let $h_0\in{\rm supp}\,\kappa$. Given $n\geq1$, let $A_n:=\{x\in X : \psi_n(x)=h_0\}$. By taking subsets of $A_n$ if necessary, we can assume that there exists $c>0$ such that $\mu(A_n)=c$ for all $n\geq1$. Again, by passing to a subsequence if necessary, we can assume that
$$
\left(\va_n|_{A_n}\right)_\ast(\mu)\to \nu',$$
where $\nu'$ is concentrated on $C$ and $\nu'(C)=c$. Take any $g_0\in {\rm supp}\,\nu'$ (note that ${\rm supp}\,\nu'\subset{\rm supp}\,\nu$). Fix $\vep>0$. Then (by the invariance of $d$ under translations)
$$
B(g_0,\frac{\vep}2)+B(h_0,\frac{\vep}2)\subset B(g_0+h_0,\vep).$$
Since $g_0\in{\rm supp}\, \nu'$, there exist $c'>0$ such that
$$
\liminf_{n\to\inf}(\varphi_n|_{A_n})_\ast \mu (B(g_0,\vep/2))\geq c',
$$
so for $n$ large enough there exists $A_n'\subset A_n$ such that $\mu(A_n')\geq c'/2$ and $\va_n(A_n')\subset B(g_0,\frac{\vep}2)$.
But $\psi_n(A_n')=\{h_0\}$, whence $(\va_n+\psi_n)(A_n')\subset B(g_0+h_0,\vep)$.

(b) Suppose now that $g_0\in{\rm supp}\,\nu$. Fix $\vep>0$. Then there exist $c>0$ and $A_n\subset X$ such that $\mu(A_n)=c$ and $\va_n(A_n)\subset B(g_0,\vep)$ for each $n\geq1$. Now, we have a partition
$$
A_n=A_{n,f_1}\cup\ldots\cup A_{n,f_k},$$
where $f_j$ is the (unique) value of $\psi_n$ on $A_{n,f_j}$, $j=1,\ldots,k$. Then, for some $1\leq m\leq k$, we have $\mu(A_{n,f_m})\geq c/k$
(passing to a subsequence if necessary).  Now, the sets $A_{n,f_m}$, $n\geq1$ ``realize'' $h_0:=f_m$. Repeat the proof of (a) to find
$A_n'\subset A_{n,f_m}$ so that for $g_0^{\prime}\in{\rm supp}\,\nu^{\prime\prime}$ (with $\left(\va_n|_{A_{n,f_m}}\right)_\ast(\mu)\to\nu^{\prime\prime}$), we have $g_0'+h_0\in {\rm supp}\,\rho$. But $d(g_0,g_0')<\vep$ whence $g_0+h_0\in{\rm supp}\,\rho$ (since the support is closed).\end{proof}

\subsection{The group of $L^\infty$-eigenvalues}
If $\cw=(W_g)_{g\in G}$ is a non-singular $G$-action on a  probability standard Borel space $\zdr$ then a character $\chi\in\widehat{G}$ is called an $L^\infty$-{\em eigenvalue} of $\cw$ if for some $0\neq f\in L^\infty\zdr$ and each $g\in G$,
$$
f\circ W_g=\chi(g)\cdot f \;\;\mbox{$\rho$-a.e.}
$$
We will denote this group by $e(\cw)$.\footnote{In contrast with the finite measure-preserving case, in the non-singular case, it may happen that $e(\cw)$ is uncountable. It is however always a Borel subgroup of $\widehat{G}$.}
When $\cw$ is ergodic, we can assume that $|f|=1$ and such an action is called {\em weakly mixing} if $e(\cw)$ consists of the trivial character solely.

It follows that for automorphisms ($\Z$-actions), say for $\tf$ considered above, the group $e(\tf)$ consists of
$c\in\bs^1$ for which $F\circ \tf=cF$ for some $F\in L^\infty(X\times G,\mu\ot\la)$. If $\va$ is ergodic then
$c\in e(\tf)$ if and only if
for some $\chi\in\widehat{G}$,
$$
\chi\circ \va=c\cdot \xi\circ T/\xi$$
for some measurable $\xi\colon X\to\bs^1$, e.g.\ \cite{Le-Pa1} (note that $F(x,g)=\ov{\xi(x)}\chi(g)$ is the corresponding eigenfunction).

Recall that for an arbitrary cocycle $\va\colon X\to G$ (over an ergodic $T\in {\rm Aut}\xbm$) the Mackey action $\cw(\va)=(W_g)_{g\in G}$
is ergodic. Moreover, we have (e.g.\ \cite{Le-Pa1})
\begin{multline}\label{warMa}
e(\cw(\va))=\Lambda(\va)\\
:=\{\chi\in\widehat{G} :  \chi\circ\va=\xi\circ T/\xi\text{ for a measurable }\xi\colon X\to\bs^1\}.
\end{multline}
Note that when $\va$ is ergodic then $\Lambda(\va)=\{\raz\}$.

\begin{Lemma}\label{l:wanih}
If $\chi\in\widehat{G}$ is an $L^\infty$-eigenvalue of the Mackey action $\cw(\va,T)$, then $\chi\in E(\va)^\perp$, i.e.\ $\chi(g)=1$ for each $g\in E(\varphi)$.
\end{Lemma}
\begin{proof} Since $\ov{\chi(E(\va))}\subset E(\chi\circ\va)$ (e.g.\ \cite{Le-Pa-Vo}; the equality holds if $\va$ is additionally regular), and $E(\chi\circ\va)=\{1\}$ since $\chi\circ\va$ is a coboundary, the result follows.\end{proof}

\subsection{Rokhlin extensions}\label{RoEx}
Assume that $T\in{\rm Aut}\xbm$ is ergodic and let $\va\colon X\to G$ be a cocycle.
Assume moreover that $G\ni g\mapsto S_g\in {\rm Aut}\ycn$ is a (measurable) $G$-representation, which we denote by $\cs$.\footnote{Measurability can also be expressed by the fact that for each $h\in L^2\ycn$, the map $G\ni g\mapsto  h\circ S_g$ is continuous (equivalently, it is weakly continuous).} We will always assume that $\cs$ is ergodic.
The $G$-action $\cs$ induces a unitary $G$-representation (Koopman representation, we will use the same notation $\cs$ to denote this representation) on $L^2\ycn$ given by $f\mapsto f\circ S_g$. We denote by $\sigma_\cs$ the maximal spectral type of $\cs$ on the subspace $L^2_0\xbm$ of $L^2\xbm$ of zero mean functions. Then, the automorphism $T_{\va,\cs}\in {\rm Aut}(X\times Y,\cb\ot\cc,\mu\ot\nu)$ given by
$$
T_{\va,\cs}(x,y)=(Tx,S_{\va(x)}(y))\text{ for }(x,y)\in X\times Y
$$
is called a \emph{Rokhlin extension} of $T$.\footnote{The map $X\ni x\mapsto S_x:=S_{\va(x)}\in {\rm Aut}(X\times Y,\mu\ot\nu)$ is a particular case of so called  Rokhlin cocycle, e.g.\ \cite{Da-Le,Le-Le,Le-Pa1,Le-Pa2}.}
Note that for each $k\in\Z$, we have
$(\tf)^k(x,y)=(T^kx,S_{\va^{(k)}(x)}(y))$,
hence
\beq\label{iteracja}
(\tfs)^k=(T^k)_{\va^{(k)},\cs}.\eeq

\begin{Prop}[\cite{Le-Le,Le-Pa2}] \label{p:ergRE}
Assume that $\cs$ is ergodic. Then $\tfs$ is ergodic if and only if $\sigma_{\cs}(\Lambda_\va)=0$.
\end{Prop}

\begin{Remark}\label{r:ergRE} It follows that if $\va$ is ergodic then $\tfs$ is ergodic whenever  $\cs$ is ergodic.\end{Remark}
We will also need the following.

\begin{Prop}[\cite{Le-Pa2}] \label{p:warwlaRE}
Assume that $\tfs$ is ergodic.  Then $c\in e(\tfs)$ if and only if for some
$\chi\in e(\cs)$, we have $\chi\circ \va=c\cdot \xi\circ T/\xi$ for some measurable $\xi\colon X\to\bs^1$.\end{Prop}

\begin{Remark}\label{r:warWL} Assume that $G=\Z$ and let $Y=\Z/2\Z=\{0,1\}$ with $S(i)=i+1$ mod~2 for $i=0,1$. Then the only non-trivial eigenvalue of $S$ is -1 (more precisely, it is the character $\chi(n)=(-1)^n$) and it follows that (assuming that $\tfs$ is ergodic) $c$ is an eigenvalue of $\tfs$ different from an eigenvalue of $T$ if and only if
$$
(-1)^{\va(\cdot)}=c\cdot\xi\circ T/\xi$$ for some measurable $\xi\colon X\to\bs^1$. Similarly, if $Sy=y+\alpha$ is an irrational rotation on $\T=[0,1)$ then $c\in e(\tfs)\setminus e(T)$ for an ergodic $\tfs$ if and only if for some $k\neq0$, we have
$$
e^{2\pi ik\va(\cdot)\alpha}=c\cdot\xi\circ T/\xi$$
for a measurable $\xi\colon X\to\bs^1$.

Clearly, if $\cs$ is weakly mixing, then the ergodicity of $\tfs$ implies $e(\tfs)=e(T)$.\end{Remark}

\begin{Prop}[\cite{Le-Le}]\label{p:uniqueE} If $T$ is uniquely ergodic, $\va\colon X\to G$ is continuous and $\tf$ is ergodic then  $\tfs$ is  uniquely ergodic whenever $\cs$ is uniquely ergodic.\end{Prop}

\subsection{Joinings} Let $T\in{\rm Aut}\xbm$ and $R\in{\rm Aut}\zdr$ be ergodic. By a {\em joining}  of $T$ and $R$ we mean any $T\times R$-invariant measure $\kappa$ on $(X\times Z,\cb\ot\cd)$ whose projections on the coordinates are $\mu$ and $\rho$, respectively. The set of joinings between $T$ and $R$ is denoted by $J(T,R)$. By $J^e(T,R)$ we denote the subset of {\em ergodic joinings}, i.e.\ the subset of those $\kappa\in J(T,R)$ for which the automorphism $T\times R\in{\rm Aut}(X\times Z,\cb\ot\cd,\kappa)$ is ergodic. Following \cite{Fu}, $T$ and $R$ are called {\em disjoint} if $J(T,R)=\{\mu\ot\rho\}$.

An obvious example of a joining is just the product measure $\mu\ot\rho$.
Assume now that $T$ and $R$ are factors of $\ov{T}\in{\rm Aut}(\ov{X},\ov{\cb},\ov{\mu})$ and $\ov{R}\in{\rm Aut}(\ov{Z},\ov{\cd},\ov{\rho})$, respectively. By abuse of notation, we can assume that $\cb\subset\ov{\cb}$ and $\cd\subset\ov{\cd}$. Assume that $\kappa\in J(T,R)$. Then the measure
$\widehat{\kappa}$ determined by
$$
\int_{\ov{X}\times\ov{Z}}F\ot G\,d\widehat{\kappa}=\int_{X\times Z}{\mathbb E}(F|\cb){\mathbb E}(G|\cd)\,d\kappa$$
is a joining of $\ov{T}$ and $\ov{R}$, i.e.\ $\widehat{\kappa}\in J(\ov{T},\ov{R})$, called the {\em relatively independent extension} of $\kappa$ \cite{Gl}. Even if $\kappa$ is ergodic, the relatively independent extension of it need not be ergodic.
Note that
\beq\label{rieCALKA}
\int_{\ov{X}\times\ov{Z}}F\ot G\,d\widehat{\kappa}=0\text{ whenever }{\mathbb E}(F|\cb)=0.\eeq

If $U\in {\rm Aut}\xbm$, $U\circ T=T\circ U$, i.e.\ if $U$ belongs to the {\em centralizer} $C(T)$ of $T$, then $U$ determines a self-joining $\mu_U\in J(T,T)$ (in fact, $\mu_U\in J^e(T,T)$) given by the formula $\mu_U(A\times B)=\mu(A\cap U^{-1}B)$ for $A,B\in\cb$.
Consider, as above, an ergodic extension $\ov{T}\in{\rm Aut}(\ov{X},\ov{\cb},\ov{\mu})$ of $T$. If the relatively independent extension of $\mu_{Id}$ is ergodic, i.e.\ if
$\widehat{\mu_{Id}}\in J^e(\ov{T},\ov{T})$ then we say that the extension
$$
(\ov{X},\ov{B},\ov{\mu},\ov{T})\to (X,\cb,\mu,T)$$
is {\em relatively weakly mixing}. We also say  that $\ov{T}$ is relatively weakly mixing over~$T$.

In what follows, we will study automorphisms of the form $\tfs$ (as in Subsection~\ref{RoEx}). They are clearly extensions of $T$.
There is a simple criterion for the relative weak mixing in this situation:

\begin{Prop} [\cite{Le-Pa1,Le-Pa2}] \label{p:rwm}
$\tfs$ is relatively weakly mixing over $T$ if and only if $\tfs$ is ergodic and $\cs\subset{\rm Aut}\ycn$ is weakly mixing. Moreover, if $T$ is disjoint from all weakly mixing automorphisms, $\va$ is ergodic and $\cs$ is mildly mixing\footnote{Mild mixing means that for no set $A\in\cb$, $0<\mu(A)<1$, we have $\liminf_{n\to\infty}\mu(T^{-n}A\triangle A)=0$.}  then $\tfs$ is disjoint from all weakly mixing automorphisms.
\end{Prop}

\subsection{Joinings between Rokhlin extensions}
Let $G$ be an lcsc Abelian group. Following \cite{Le-Me-Na}, we will consider {\em algebraic couplings} of $G$ which are subgroups $\ch\subset G\times G$ whose projections on both coordinates are dense. Assume that $\cs=(S_g)_{g\in G}$ is an ergodic $G$-action in ${\rm Aut}\ycn$. By $\cs\ot\cs$ we denote the product $G\times G$-action: $\cs\ot\cs=(S_g\times S_{g'})_{(g,g')\in G\times G}$. It is considered with product measure $\nu\ot\nu$.\footnote{In fact, by ergodicity, $\cm(Y\times Y;G\times G)=\{\nu\ot\nu\}$ \cite{Le-Me-Na}.} By $(\cs\ot\cs)|_{\ch}$ we denote the restriction of the product action to $\ch$.
Let
$$\begin{array}{c}
\cm(Y\times Y;\ch):=\\
\{\rho\in\cp(Y\times Y) : \rho\text{ is }(\cs\ot\cs)|_{\ch}-\text{invariant and both projections of }\rho\text{ are }\nu\}.\end{array}$$
Then $\cm(Y\times Y;\ch)$ is a simplex whose set of extremal points is the set $\cm^e(Y\times Y;\ch)$ of ergodic measures in $\cm(Y\times Y;\ch)$.
Denote by $\sigma_{\cs}$ the maximal spectral type of $\cs$ on $L^2_0\ycn$.

\begin{Prop}[cf.\ \cite{Le-Me-Na}, Proposition~8]\label{p:cherg} Assume that $\cs$ is ergodic and $\ch\subset G\times G$ is an algebraic coupling of $G$.
Consider $((\cs\ot\cs)|_{\ch},\nu\ot\nu)$. The following conditions are equivalent:
\begin{enumerate}
\item
$(\cs\ot\cs)|_{\ch}$  is ergodic.
\item
$\left(\sigma_\cs\ot\sigma_\cs+\sigma_\cs\ot\delta_\raz+\delta_\raz\ot\sigma_{\cs}\right)(\ch^\perp)=0$.
\item
If $\chi,\eta\in\widehat{G}$ are eigenvalues for $\cs$ then
$$(\chi,\eta)\in\ch^{\perp}\; \Leftrightarrow\; \chi=\raz=\eta.$$
\end{enumerate}
\end{Prop}
\begin{proof} Let $p:\widehat{G}\times\widehat{G}\to\widehat{\ch}$ be the dual map corresponding to the embedding
of $\ch$ into $G\times G$. The kernel of this homomorphism
is
$$
{\rm ker}\, p=\ch^{\perp}=\{\xi\in\widehat{G\times G} : 
\xi(\ch)=\{1\}\}.
$$
The maximal spectral type of $(\cs\otimes\cs)|_{\ch}$ on $L^2_0(Y\times Y,\nu\ot\nu)$ is equal to $$
p_\ast\left(\sigma_\cs\ot\sigma_\cs+\sigma_\cs\ot\delta_\raz+\delta_\raz\ot\sigma_{\cs}\right).$$
Hence we can obtain the trivial character in the image of $p$ only for an eigenvalue for $\cs\otimes\cs$, belonging to $\ch^\perp$ (and such are only products of eigenvalues for $\cs$).
\end{proof}

\begin{Cor}\label{c:cherg1}
If, additionally, $\ch$ is cocompact and
$$
\left(\sigma_\cs\ot\sigma_\cs+\sigma_\cs\ot\delta_\raz+
\delta_\raz\ot\sigma_{\cs}\right)(\ch^\perp)=0$$
then $\cm(Y\times Y;\ch)=\{\nu\ot\nu\}$.\end{Cor}
\begin{proof}
Let $\rho\in \cm^e(Y\times Y;\ch)$. Then also $(S_{g}\times S_{g'})_\ast(\rho)\in\cm^e(Y\times Y;\ch)$ for each $g,g'\in G$.
Define
\beq\label{defkoz}
\widetilde{\rho}:=\int_{(G\times G)/\ch}(S_g\times S_{g'})_\ast\rho\,d((g,g')\ch).\eeq
Then $\widetilde{\rho}$ is an $\cs\otimes\cs$-invariant measure, whence $\widetilde{\rho}=\nu\ot\nu$.
But by our assumption and Proposition~\ref{p:cherg}, it follows that
$$
((\cs\ot\cs)|_{\ch},\nu\ot\nu)\text{ is ergodic}.$$
It follows that in the integral representation~\eqref{defkoz} of $\tilde{\rho}=\nu\ot\nu$ all the measures are $\ch$-invariant and ergodic and therefore they must be a.e.\  equal to $\nu\ot\nu$ (since $\nu\ot\nu\in\cm^e(Y\times Y;\ch)$). Finally, if $\left(S_g\times S_{g'}\right)_\ast(\rho)=\nu\ot\nu$ for some $(g,g')\in G\times G$ then $\rho=\nu\ot\nu$ and the result follows.\end{proof}

Assume now that $T\in{\rm Aut}\xbm$ is ergodic and let $\va,\psi\colon X\to G$ be cocycles. Fix $\la\in J^e_2(T)$ and consider
$$
\va\times\psi:(X\times X,\la)\to G\times G.$$
To underline that we have dependence on $\la$, we will also write $(\va\times\psi)_\la$.
Denote by $J(\tfs,T_{\psi,\cs};\la)$ the subset of $J(\tfs,T_{\psi,\cs})$ consisting of these joinings whose restriction to $X\times X$ equals $\la$.
 Let
\beq\label{Hlambda}
\ch_\la:=E((\va\times\psi)_\la).\eeq
Proceeding now as in \cite{Le-Me-Na}, we obtain the following result.

\begin{Th}\label{t:isoLambda} Under the above assumptions,
suppose that $\va,\psi\colon X\to G$ are ergodic and let $(\va\times\psi)_\la$ be regular. Then there exists an affine isomorphism
$$
\Lambda_\la:J(\tfs,T_{\psi,\cs};\la)\to \cm(Y\times Y;\ch_\la).$$\end{Th}

Directly from Proposition~\ref{p:cherg}, Corollary~\ref{c:cherg1} and  Theorem~\ref{t:isoLambda}, we obtain the following.

\begin{Cor}\label{c:jedynosc}
If additionally $\ch_\la$ is cocompact and
$$
\left(\sigma_\cs\ot\sigma_\cs+\sigma_\cs\ot\delta_\raz+\delta_\raz\ot\sigma_{\cs}\right)(\ch^\perp)=0$$
then
$
J(\tfs,T_{\psi,\cs};\la)=\{\la\ot\nu\ot\nu\}$.
\end{Cor}

\subsection{Ergodic sequences and suspensions}\label{SecSusp}
Let $G$ be an lcsc  Abelian group.
\begin{Def}[see e.g.\ \cite{Be-Bo-Bo}]
A sequence $(a_n)\subset G$   is called {\em ergodic}  if for each ergodic $G$-action  $\cs=(S_g)_{g\in G}\subset{\rm Aut}\ycn$, we have $$\frac1N\sum_{n\leq N}f\circ S_{a_n}\to \int_Y f\,d\nu \;\;\mbox{in}\;\; L^2\ycn$$ for each $f\in L^2\ycn$.
\end{Def}
To obtain ergodic sequence, one can use the following result.\footnote{For the ergodicity of sequences $([n^c])$ ($0<c<1$), see \cite{Be-Bo-Bo}, and for the case of slowly increasing sequence, see \cite{Bo-Wi}.}
\begin{Prop}[see e.g.\  \cite{Le-Le-Pa-Vo-Wi}] \label{ergseq}
If  $T$ is uniquely ergodic, $\va\colon X\to G$ is continuous and $\Lambda(\va)=\{\raz\}$, i.e.\ the Mackey action $\cw(\va)$ is weakly mixing, then for each $x\in X$, the sequence $(\va^{(n)}(x))$ is ergodic. In particular, the result holds if $\va$ is ergodic.
\end{Prop}
\begin{proof} Since $\va$ is continuous, so is $\chi\circ\va$ for each $\chi\in\widehat{G}$. Moreover, since $\Lambda(\va)=\{1\}$, for each $1\neq\chi\in\widehat{G}$ the (compact) group extension $T_{\chi\circ\va}$ defined on $(X\times\bs^1,\mu\ot\la_{\bs^1})$ is uniquely ergodic \cite{Fu1}. It follows that for $G(x,z)=z$ and every $(x_0,z_0)$, we have
$$
\frac1N\sum_{n\leq N}G((T_{\chi\circ\va})^n(x_0,z_0))\to\int G\,d(\mu\ot\la_{\bs^1})=0.$$
Therefore, whenever $\raz\neq\chi\in\widehat{G}$,  for all $x\in X$, we have
\beq\label{punktowo}
\frac1N\sum_{n\leq N}\chi(\va^{(n)}(x))\to 0.\eeq
Assume that $f\in L^2_0\ycn$. By the spectral theorem, for each $x\in X$, we have
\beq\label{spectral}
\left\|\frac1N\sum_{n\leq N}f\circ S_{\va^{(n)}(x)}\right\|_{L^2\ycn}=\left\|\frac1N\sum_{n\leq N}\chi(\va^{(n)}(x))\right\|_{L^2(\widehat{G},\sigma_f)},\eeq
where $\sigma_f$ stands for the {\em spectral measure} of $f$ (with respect to the Koopman representation of $\cs$), i.e.\ the only measure on $\widehat{G}$ determined by
$$
\widehat{\sigma_f}(g):=\int_{\widehat{G}}\chi(g)\,d\sigma_f(\chi)=\int_Y f\circ S_g\cdot\ov{f}\,d\nu,\;\;g\in G.$$
Since $f\in L^2_0\ycn$ and $\cs$ is ergodic, $\sigma_f$ has no atom at the trivial character~$\raz$, and the result follows from~\eqref{punktowo}.
\end{proof}

\begin{Def}[\cite{Fu1}, p.\ 72]
A sequence $(a_n)\subset G$ is a  {\em Poincar\'e sequence} if for each ergodic $G$-action $\cs=(S_g)_{g\in G}\subset{\rm Aut}\ycn$ and $A\subset Y$ of positive measure there exists $n\geq1$ such that $\nu(S_{a_n}A\cap A)$ is positive.
\end{Def}

\begin{Remark}
Each ergodic sequence is a Poincar\'e sequence.
\end{Remark}

In what follows, we will be particularly interested in ergodic sequences for flows ($\R$-actions) and automorphisms ($\Z$-actions).
If $(k_n)\subset\Z$ is an ergodic sequence for automorphisms then clearly it will be an ergodic sequence for all flows whose time~1 automorphism is ergodic. 

To see some relation in the opposite direction  consider the following construction.
Assume that $R\in{\rm Aut}\zdr$ is ergodic.
Consider $\va:Z\to\R$,
$\va(z)=1$ and let $\la$ be a probability measure equivalent to $\la_{\R}$. It $F:Z\times\R\to\C$ is measurable and $F\circ R_\va=F$ $\mu\ot\la$-a.e.\ then $F$ is entirely determined by $F|_{Z\times [0,1)}$; indeed $F(z,s):=F(R^{-n}z,s-n)$, where $n\in\Z$ is unique so that $0\leq s-n<1$. It easily follows that the Mackey action of $R_\va$ is the {\em suspension flow} \cite{Co-Fo-Si} $\widetilde{R}=(\widetilde{R}_t)_{t\in\R}$ of $R$, which  acts on the space
 $\widetilde{Z}:=Z\times[0,1]/\sim$, where
$(z,1)\sim(Rz,0)$, by the formula
\beq\label{zaw1}
\widetilde{R}_t(z,s)=(R^{[t+s]}z,\{s+t\}).\eeq
This flow preserves the product measure denoted by $\widetilde{\mu}$.

\begin{Remark}\label{r:ude} Note that if $(a_n)\subset\R$ is ergodic then $(a_n)$ is uniformly distributed modulo~1. Indeed, a sequence $(a_n)$ is ergodic then $\frac1N\sum_{n\leq N}e^{2\pi i\theta a_n}\to0$ for each $\theta\in\R\setminus\{0\}$ (cf.\ the proof of Proposition~\ref{ergseq}). Hence $(\{a_n\})$ is uniformly distributed on $\T$ by the Weyl criterion.\end{Remark}

\begin{Remark}\label{krz} Assume that $(a_n)\subset\R$ is an ergodic sequence. Then $([a_n])$ is a Poincar\'e sequence for each ergodic automorphism $R$. Indeed, if we take a set $A\subset Z$ of positive measure then for some $a_n$ (arbitrarily large)
$$\widetilde{\mu}\left(\widetilde{R}_{-a_n}(A\times [0,1/2))\cap
(A\times[1/2,1))\right)>0.$$ It follows that $\mu(R^{-[a_n]}A\cap A)>0$.

As shown in \cite{Les}, it is however possible that $(a_n)\subset\R$ is an ergodic sequence for flows but $([a_n])$ is not an ergodic sequence for automorphisms.\end{Remark}

\section{The AOP property and multiplicative functions}

\subsection{Automorphisms with the AOP property}\label{SectAOP}
Let $\mathscr{P}$ denote the set of prime numbers. Following \cite{Ab-Le-Ru}, $T\in{\rm Aut}\xbm$ is said to have {\em asymptotically orthogonal powers} (AOP) if
for each $f,g\in L^2_0\xbm$, we have
\beq\label{aop2016}
\limsup_{r\neq s,r,s\in\mathscr{P},r,s\to\infty} \sup_{\kappa\in J^e(T^r,T^s)}\left|\int_{X\times X}f\ot g \,d\kappa\right|=0.\eeq
The condition above means that, intuitively, the prime powers of an automorphism become more and more disjoint. Clearly, the AOP property is inherited by factors, therefore the AOP property of $T$ implies its total ergodicity.\footnote{Ergodic rotations of finite non-trivial cyclic groups do not have the AOP property.}
As shown in \cite{Ab-Le-Ru}, all totally ergodic rotations enjoy the AOP property.\footnote{Note that such rotations can have all non-zero powers isomorphic. However, when $r,s$ are growing, the graphs of such isomorphisms which yield ergodic joinings between $T^r$ and $T^s$ are more and more ``mixing''.} For larger classes of AOP automorphisms, see \cite{Ab-Le-Ru} and \cite{Fl-Fr-Ku-Le}.

\subsection{Orthogonality of AOP observables with multiplicative functions}
We will consider {\em arithmetic functions}, i.e.\ $\bfu\colon\N\to\C$, $\bfu=(\bfu(n))_{n\in\N}$. Recall that such a function is {\em multiplicative} if $\bfu(mn)=\bfu(m)\bfu(n)$ whenever $(m,n)=1$. Many basic arithmetic functions as the M\"obius function $\mob$ or the Liouville $\lio$ function are multiplicative.\footnote{Dirichlet characters are other examples of multiplicative functions. Also, for each $t\in\R$, the function $n\mapsto n^{it}$ is multiplicative.}
It is a general problem in number theory to show that a given sequence $(a_n)\subset\C$ is {\em orthogonal} to a multiplicative function $\bfu$, i.e.
\beq\label{defort}
\frac1N\sum_{n\leq N}a_n\bfu(n)\to0\text{ when }N\to\infty.
\eeq
The basic criterion to show such an orthogonality  is  the following version (see \cite{Ab-Le-Ru}) of the K\'atai-Bourgain-Sarnak-Ziegler (KBSZ) criterion.

\begin{Prop}[\cite{Bo-Sa-Zi,Ka}] \label{kbszcrit}
Assume that $(a_n)\subset \C$ is bounded and
$$
\limsup_{r\neq s,r,s\in\mathscr{P},r,s\to\infty}\left(\limsup_{N\to\infty}\left|\frac1N\sum_{n\leq N}a_{rn}\ov{a}_{sn}\right|\right)=0.
$$
Then $(a_n)$ is orthogonal to any multiplicative function $\bfu$, $|\bfu|\leq1$.
\end{Prop}

Given a homeomorphism $T$ of a compact metric space $X$, $f\in C(X)$ and $x\in X$, we are interested in the orthogonality of the observable $(f(T^nx))_{n\geq0}$ with $\bfu$. For example, Sarnak's conjecture (cf.\ \eqref{mob00}) claims that whenever $T$ has topological entropy zero then all observables as above are orthogonal to the M\"obius function.

To see a relationship of the KBSZ criterion with joinings, assume that $T$ is a uniquely ergodic homeomorphism of a compact metric space $X$ (the unique $T$-invariant measure is denoted by $\mu$). Assume that $(X,\mu,T)$ is totally ergodic and let $x\in X$. Consider the sequence
\beq\label{sequence}
\frac1N\sum_{n\leq N}\delta_{(T^r\times T^s)^n(x,x)},\;N\geq1,
\eeq
of probability measures on $X\times X$. Any accumulation point $\kappa$ of such  measures is $T^r\times T^s$-invariant, and by the unique ergodicity of $T$ (and the total ergodicity assumption), $\kappa\in J(T^r,T^s)$. It follows that if we are able to show that $T^r$ and $T^s$ are disjoint for $r\neq s$ sufficiently large then
$$
\frac1N\sum_{n\leq N}\delta_{(T^r\times T^s)^n(x,x)}\to\mu\ot\mu,
$$
whence
\begin{multline*}
\frac1N\sum_{n\leq N}f(T^{rn}x)\ov{f(T^{sn}x)}\\
=\int f\ot \ov{f}\,d\left(\frac1N\sum_{n\leq N}\delta_{(T^r\times T^s)^n(x,x)}\right)\to\int f\ot\ov{f}\,d(\mu\ot\mu)=0
\end{multline*}
whenever $f\in L^2_0\xbm$. It follows that the KBSZ criterion applies.

It turns out however that to apply Proposition~\ref{kbszcrit} we do not need to obtain that the only accumulation point  of the sequence~\eqref{sequence} is product measure. We just need to show  that when $r\neq s$ are large, such accumulation points are close (in the weak topology) to product measure. This is guaranteed if the AOP property holds (see \cite{Ab-Le-Ru} for details). In fact, the AOP property implies even more:

\begin{Prop} [\cite{Ab-Ku-Le-Ru1}] \label{AOPtoMOMO} Assume that
$T$ is a uniquely ergodic homeomorphism of a compact metric space
$X$ such that $(X,\mu,T)$ enjoys the AOP property. Then
for each multiplicative $\bfu$, $|\bfu|\leq1$, each  sequence $(b_k)$ of natural numbers with $b_{k+1}-b_k\to\infty$ and each choice of $x_k\in X$, $k\geq1$, and $f\in C(X)$, $\int f\,d\mu=0$, we have
$$
\frac1{b_{K+1}}\sum_{k\leq K}\left|\sum_{b_k\leq n<b_{k+1}}f(T^nx_k)\bfu(n)\right|\to0\text{ when }K\to\infty.$$\end{Prop}

The above property is referred to as the {\em strong} $\bfu$-MOMO property (in other words AOP implies the strong MOMO property).
As noticed already in \cite{Ab-Le-Ru}, from the strong $\bfu$-MOMO property, we obtain the following:
\beq\label{shortM}
\frac1{M}\sum_{M\leq m<2M}\left|\frac1H\sum_{m\leq h<m+H}f(T^hx)\bfu(n)\right|\to0\eeq
when $H\to\infty$, $H/M\to0$,
for each $x\in X$ and $f,\bfu$ as above.

We can extend the notion of strong $\bfu$-MOMO property to arbitrary lcsc Abelian groups by considering sequences of group elements.
Fix a sequence $(a_n)\subset G$.

\begin{Def}\label{d:momoG} A uniquely ergodic $G$-action $\cs$ of  a compact metric space $Y$ (the unique invariant measure is denoted by $\nu$) is said to satisfy {\em the strong $\bfu$-MOMO property along} $(a_n)$ if
for each  sequence $(b_k)$ of natural numbers with $b_{k+1}-b_k\to\infty$ and each choice of $y_k\in Y$, $k\geq1$, and $f\in C(Y)$, $\int f\,d\nu=0$, we have
$$
\frac1{b_{K+1}}\sum_{k\leq K}\left|\sum_{b_k\leq n<b_{k+1}}f(S_{a_n}y_k)\bfu(n)\right|\to0\text{ when }K\to\infty.$$
\end{Def}

\section{Lifting the AOP property by Rokhlin cocycles}\label{se:li}
\subsection{The main result}
Assume that $T\in{\rm Aut}\xbm$ has the AOP property. Assume that $G$ is an lcsc Abelian group and let $\va\colon X\to G$ be an ergodic cocycle. We are interested in a sufficient condition on $T$ and $\va$ to ensure that  the Rokhlin extensions $\tfs$ have the AOP property for all ergodic $G$-actions $\cs=(S_g)_{g\in G}$.

\begin{Th}\label{jm} Assume that $T$ has the AOP property.
Assume moreover that for each $r\neq s$, $r,s\in\cp$ and arbitrary $\eta\in J^e(T^r,T^s)$:
\begin{align}
&\text{the group extension }(T_\va)^r\times T^s\text{ is ergodic }\mbox{(over $(T^r\times T^s,\eta)$)};\label{E4}\\
&\text{the Mackey action }\cw((\va^{(r)}\times \va^{(s)},T^r\times T^s,\eta)\text{ is weakly mixing.}\label{E5}
\end{align}
Let $\cs=(S_g)_{g\in G}$ be an ergodic $G$-action on $\ycn$.
Then $\tfs$ has the AOP property.
\end{Th}

\subsection{Remarks on total ergodicity}
Recall that the AOP property implies total ergodicity. We will show now that each of~\eqref{E4} and~\eqref{E5} also implies total ergodicity of $\tfs$ for each ergodic $\cs$.

\begin{Prop}
Condition~\eqref{E4} implies the total ergodicity of $\tfs$ for each ergodic $\cs$.
\end{Prop}
\begin{proof}
Suppose that for some ergodic $\cs$ acting on $\ycn$, $\tfs$ is not totally ergodic. Then for some prime number $r$, $(\tfs)^r$ is not ergodic. Since $(\tfs)^r=(T^r)_{\va^{(r)},\cs}$ and $(\tfs)^r$ is a factor of $(T^r\times T^s,\eta)_{\va^{(r)},\cs}$, we obtain a contradiction to~\eqref{E4}, cf.\ Remark~\ref{r:ergRE}.
\end{proof}
Note also that~\eqref{E4} implies the total ergodicity of $\tfs$ itself. Indeed, the following holds:
\begin{Lemma} Suppose that $R$ is a non-singular  automorphism of a  probability standard Borel space $\zdr$.
Suppose that $R^k$ is ergodic for all $k\in\mathscr{P}$. Then $R$ is totally ergodic.
\end{Lemma}
\begin{proof}
Suppose that $R^n$ is not ergodic for some $n>1$ and let $n$ be the smallest number with this property. Let $A$ be non-trivial such that $R^nA=A$. Let $J\subset \{0,\dots, n-1\}$ be a maximal subset such that $B:=\bigcap_{j\in J}R^jA$ is of positive measure. Clearly, $R^nB=B$ and we claim that $\{R^iB\}_{i=0}^{n-1}$ is a Rokhlin tower. Indeed, all we need to show is that $\mu(B\cap R^iB)=0$ for $i=1,\ldots,n-1$. In other words, $J\neq J+i \bmod n$. However, if $J=J+i\bmod n$ then $R^iB=B$ which contradicts the ergodicity of $R^i$ by the choice of $n$ (recall that $1\leq i<n$). Let $p$ be a prime dividing $n$. Then $B':=B\cup R^pB\cup\dots \cup R^{n-p}B$ is a non-trivial $R^p$-invariant set which contradicts the ergodicity of $R^p$.
\end{proof}

\begin{Prop}
Condition~\eqref{E5} implies the total ergodicity of $\tfs$ for each ergodic $\cs$.
\end{Prop}
\begin{proof}
Suppose that $e^{2\pi i/k}$ is an eigenvalue of $\tfs$. Take any primes $r\neq s$ such that $r=s$ mod~$k$.  Since $\va$ is ergodic, so is $\tfs$ and by Proposition~\ref{p:warwlaRE}, for some $\chi\in\widehat{G}$, we have $\chi\circ\va=e^{2\pi i/k}\cdot\xi\circ T/\xi$ for some measurable $\xi\colon X\to\bs^1$. Hence
$$
\chi\circ \va^{(r)}=e^{2\pi ir/k}\cdot \xi\circ T^r/\xi\text{ and }
\chi\circ \va^{(s)}=e^{2\pi is/k}\cdot \xi\circ T^s/\xi.
$$
If we fix any $\eta\in J^e(T^r,T^s)$ then we obtain
$$
\chi\circ\ov{\chi}(\va^{(r)}(x),\va^{(s)}(x'))=\xi\ot\ov{\xi}\circ (T^r\times T^s)(x,x')/\xi\ot\ov{\xi}(x,x')$$
for $\eta$-a.e.\ $(x,x')\in X\times X$ which contradicts~\eqref{E5}, cf.~\eqref{warMa}.
\end{proof}

\subsection{More on non-singular actions, cocycles, and Mackey actions}
Throughout, we assume that $G$ is an lcsc Abelian group.

\begin{Lemma} \label{l1}
Assume that  $\cw=(W_g)_{g\in G}$ is a  non-singular action of $G$ on a  probability standard Borel space $(C,{\cal J}, \kappa)$.  Let $H\subset G$ be a closed subgroup and assume that the sub-action $(W_h)_{h\in H}$ is ergodic. Let
$\chi\in\widehat{H}$ be an $L^\infty$-eigenvalue for $(W_h)_{h\in H}$.
Then there is an $L^\infty$-eigenvalue $\widetilde{\chi}\in\widehat{G}$ of $(W_g)_{g\in G}$ such that $\widetilde{\chi}|_{H}=\chi$.\end{Lemma}
\begin{proof} By assumption, there is a measurable function $F$, $|F|=1$, such that $F\circ W_h=\chi(h)\cdot F$ for each $h\in H$. Then, for each $g\in G$,  $F\circ W_g$ is also an eigenfunction corresponding to $\chi$, so by ergodicity, $F\circ W_g=c_g\cdot F$ for a (unique) number $c_g$, $|c_g|=1$. Now, $g\mapsto c_g=F\circ W_g/F$ is a measurable homomorphism, whence there exists $\widetilde{\chi}\in\widehat{G}$ such that $\widetilde{\chi}(g)=c_g$, $g\in G$. Since $c_h=\chi(h)$ for $h\in H$, the result follows.
\end{proof}

Immediately from Lemma~\ref{l1} we obtain the following.

\begin{Lemma}\label{l2}Assume that $\cw=(W_g)_{g\in G}$ is a  weakly mixing non-singular $G$-action. Let $H\subset G$ be a closed subgroup such that the subaction $(W_h)_{h\in G}$ is ergodic. Then $(W_h)_{h\in H}$ is weakly mixing.\end{Lemma}

We will also need the following classical result.

\begin{Lemma}\label{l4}
Assume that $\cs=(S_g)_{g\in G}\subset{\rm Aut}\ycn$ is ergodic. Let $\cw=(W_g)_{g\in G}$ be a  non-singular,  weakly mixing $G$-action on $(C,{\cal J},\kappa)$. Then the $G$-action
$\cs\times\cw:=(S_g\times W_g)_{g\in G}$ on $(Y\times C,\cc\ot{\cal J},\nu\ot\kappa)$ is ergodic.
\end{Lemma}\begin{proof}
We use Keane's criterion (see~\cite{Aa}, Theorem 2.7.1) for the ergodicity of the direct product of an ergodic finite measure-preserving action  and an ergodic non-singular action. If by $\sigma_\cs$ we denote the maximal spectral type of $\cs$ on $L^2_0\ycn$ then the product $G$-action $\cs\times \cw$ is ergodic if and only if $\sigma_\cs(e(\cw))=0$. In our case $e(\cw)=\{\raz\}$ and since $\cs$ is ergodic, $\sigma_\cs$ has no atom at the trivial character.
\end{proof}

\begin{Def}[\cite{Da-Le}]
Given a non-singular $G$-action $\cw$ on $(C,{\cal J},\kappa)$ and a $G$-invariant $\sigma$-algebra $\ca\subset{\cal J}$, the corresponding extension
\begin{equation}\label{e:ext}
\pi\colon (C,\cj, \kappa,\cw)\to (C/\ca,\ca,\kappa|_{\ca},\cw|_{C/\ca})
\end{equation}
is called \emph{relatively finite measure-preserving} (rfmp for short) if the Radon-Nikodym derivative
$d\kappa\circ W_g/d\kappa$ is $\ca$-measurable for each $g\in G$.
\end{Def}
\begin{Remark}[\cite{Da-Le}]\label{u:2}
If~\eqref{e:ext} is rfmp then
$$
\frac{d(\kappa\circ W_{g})}{d\kappa}(c)=\frac{d(\kappa|_{\ca}\circ  W_{g})}{d(\kappa|_{\ca})}(\pi(c))
$$
for $\kappa$-a.e.\ $c\in C$ and every $g\in G$.
\end{Remark}
\begin{Remark}\label{E1}
Let $H\subset G$ be a closed subgroup. It follows by Remark~\ref{u:2} that if a $G$-extension is rfmp, then the corresponding $H$-extension is also rfmp.
\end{Remark}
\begin{Remark}[\cite{Da-Le}] \label{u:3}
Let $\kappa=\int_{C/\ca}\kappa_{\ov{c}}\,d\ov{\kappa}(\ov{c})$ (where $\ov{\kappa}=\kappa|_{\ca}$). It follows by Remark~\ref{u:2} that extension \eqref{e:ext} is rfmp if and only if
$
\kappa_{W_g(\pi(c))}\circ W_g=\kappa_{\pi(c)}$ for $\nu$-a.e.\ $c\in C$ and for all $g\in G$.

Note that if the extension \eqref{e:ext} is rfmp and the action $\cw|_{C/\ca}$ preserves the measure $\kappa|_{\ca}$, then
$\cw$ also preserves the measure $\kappa$.
\end{Remark}

\begin{Lemma} [see Lemma 5.3 in \cite{Da-Le}] \label{l5} Assume that $\cw=(W_g)_{g\in G}$ is an ergodic, non-singular $G$-action on $(C,{\cal J},\kappa)$. Let $(Z,\cd)$ be a standard Borel space. Assume that $\gamma$ is a probability measure on $(C\times Z,{\cal J}\ot\cd)$ whose projection on the $C$-coordinate is $\nu$ and which is $(W_g\times Id)_{g\in G}$-quasi-invariant. Assume moreover that the extension
$$
(C\times Z,{\cal J}\ot\cd,\gamma,\cw\times Id)\to (C,{\cal J},\kappa,\cw)$$
is rfmp. Then $\gamma=\kappa\otimes\rho$ for a probability measure $\rho$ on $(Z,\cd)$.\end{Lemma}
\begin{proof}
Writing $\gamma=\int_{C}\gamma_{c}\,d\kappa(c)$ and using Remark~\ref{u:3}, we obtain that $\gamma_{W_gc}=\gamma_c$ for $\kappa$-a.e.\ $c\in C$ (and all $g\in G$). Since the map $c\mapsto \gamma_c$ is measurable, the result follows by the ergodicity of $\cw$.
\end{proof}

Let $(C,{\cal J})$ be a standard Borel space. Denote by $\cp(C,\cj)$ the space of probability measures on $(C,\cj)$.
Given a Borel $G$-action $\cw$ on  $(C,{\cal J})$ and a $G$-invariant $\sigma$-algebra $\ca\subset{\cal J}$ with a quasi-invariant probability measure $\kappa$, we set
\begin{multline*}
\cp(\cw,\ca,\kappa):=\{\tilde{\kappa}\in\cp(C,{\cal J}) : \tilde\kappa\text{ is quasi-invariant, }\tilde{\kappa}|_{\ca}=\kappa\\ \text{ and the extension }(C,{\cal J},\widetilde{\kappa},\cw)\to (C/\ca,\ca,\kappa,\cw)\text{ is rfmp}\}.
\end{multline*}

\begin{Lemma}\label{l6}
Assume that $\cs^{(i)}=(S^{(i)}_g)_{g\in G}\subset{\rm Aut}(Y_i,\cc_i,\nu_i)$ is ergodic for $i=1,2$. Assume moreover that $\cw=(W_{(g_1,g_2)})_{(g_1,g_2)\in G\times G}$ is a non-singular, weakly mixing  $G\times G$-action on a  probability standard Borel space $(\widetilde{C},\widetilde{\cal J},\widetilde{\kappa})$. If the $G$-subaction $\cw|_{\{0\}\times G}$ is ergodic then
$$
\{\rho\in\cp((\cs^{(1)}\ot \cs^{(2)})\times \cw,\widetilde{\cal J},\widetilde{\kappa}): \rho|_{\cc_2\ot\widetilde{\cal J}}=\nu_2\ot\widetilde{\kappa},\;\rho|_{\cc_1}=\nu_1\}=
\{\nu_1\ot\nu_2\ot\widetilde{\kappa}\},
$$
where $\cs^{(1)}\ot \cs^{(2)}$ stands for the $G\times G$-action $(g_1,g_2)\mapsto S^{(1)}_{g_1}\times S^{(2)}_{g_2}$ on the space $(Y_1\times Y_2,\cc_1\ot\cc_2,\nu_1\ot\nu_2)$.
\end{Lemma}
\begin{proof}
Take $\rho\in\cp((\cs^{(1)}\ot \cs^{(2)})\times \cw,\widetilde{\cal J},\widetilde{\kappa})$ such that $\rho|_{\cc_2\ot\widetilde{\cal J}}=\nu_2\ot\widetilde{\kappa},\;\rho|_{\cc_1}=\nu_1$. By Lemma~\ref{l2}, the subaction $\cw|_{\{0\}\times G}$ is weakly mixing. Therefore, since $\cs^{(2)}$ is ergodic, it follows by Lemma~\ref{l4} that the (non-singular) $G$-action $\cs^{(2)}\times \left(\cw|_{\{0\}\times G}\right)$ on $(Y_2\times\widetilde{C},\nu_2\otimes\widetilde{\kappa})=(Y_2\times\widetilde{C},\rho|_{\cc_2\ot\widetilde{\cal J}})$ is ergodic. In view of Remark~\ref{E1}, the $G$-extension
$(Y_1\times Y_2\times\widetilde{C},\rho, Id\ot \cs^{(2)}\times \left(\cw|_{\{0\}\times G}\right))\to (\widetilde{C},\widetilde{\kappa},W|_{\{0\}\times G})$ is rfmp. By Lemma~\ref{l5} (applied to $\cw':=(\cs^{(2)}\times \left(\cw|_{\{0\}\times G}\right),\rho|_{\cc_2\ot\widetilde{\cal J}})$ and $\gamma=\rho$), we obtain that
$\rho=\nu'\ot (\nu_2\ot\widetilde{\kappa})$. The result follows now from the equality $\rho|_{\cc}=\nu_1$.
\end{proof}

Assume that $T\in{\rm Aut}\xbm$ is  ergodic. Let $\theta\colon X\to J$ be a cocycle with values in an lcsc Abelian group $J$.
Let $\cw(\theta)=(W(\theta)_j)_{j\in J}$  denote the associated Mackey $J$-action on the space $(C_\theta,{\cal J}_\theta,\kappa_\theta)$ of ergodic components of $T_\theta$.

Let ${\cal R}=(R_j)_{j\in J}\subset{\rm Aut}\zdr$ be an ergodic $J$-action.
We will need the following.
\begin{Lemma}[Prop.\ 6.1, Remark 6.2 \cite{Da-Le}, Prop.\ 2.1 \cite{Le-Pa1}] \label{dlp}
 There exists an affine isomorphism
 $A\colon \cp(T_{\theta,{\cal R}},\cb,\mu)\to \cp(\cw(\theta)\times {\cal R}, {\cal J}_\theta,\kappa_\theta)$ such that whenever ${\cal E}\subset\cd$ is an ${\cal R}$-invariant sub-$\sigma$-algebra, $\eta\in \cp(T_{\theta,{\cal R}},\cb,\mu)$ satisfies $\eta|_{\cb\ot{\cal E}}=\mu\ot\nu_1$, where $\nu_1$ is ${\cal R}$-invariant, then
 \beq\label{e3}
 A(\eta)|_{{\cal J}_\theta\ot{\cal E}}=\kappa_\theta\ot\nu_1.\eeq
\end{Lemma}

\begin{Remark}\label{lele} If $\theta$ above is ergodic then $\mu\ot\rho$ is the only $T_{\theta, {\cal R}}$-invariant measure whose projections on $X$ and $Z$ are $\mu$ and $\rho$, respectively \cite{Le-Le}.\end{Remark}

For $\va,\psi\colon X\to G$, we define $\va\times\psi\colon X\to G\times G$ by
$$
(\va\times \psi)(x):=(\va(x),\psi(x))\text{ for } x\in X.
$$

\begin{Lemma}\label{l3}
Let $\va,\psi\colon X\to G$ be ergodic cocycles. Then:
\begin{enumerate}[(i)]
\item
The $G$-subactions $\cw(\va\times\psi)|_{\{0\}\times G}$ and $\cw(\va\times\psi)|_{G\times\{0\}}$ are ergodic.\footnote{In what follows, see Lemma~\ref{l8} and the proof of Theorem~\ref{jm} below, we will consider the situation in which $T$ is replaced by $(T^r\times T^s,\eta)$ (with $\eta\in J^e(T^r,T^s)$) and then we will consider the group extension $(T^r\times T^s)_{\va^{(r)}\times\va^{(s)}}$. Note that then the assumption~\eqref{E4} will imply the validity of (i) with $T$ replaced by $(T^r\times T^s,\eta)$.}
\item
$\cw(\va\times\psi)$ is weakly mixing if and only if the only character $(\chi,\eta)\in \widehat{G}\times\widehat{G}$ for which there exists
a measurable $\xi\colon X\to\bs^1$ satisfying
$$
\chi(\va(x))\eta(\psi(x))=\xi(Tx)/\xi(x)
$$
is the trivial character ($\chi=\eta=1$).
\item If $\cw(\va\times\psi)$ is weakly mixing, so are the subactions $\cw(\va\times\psi)|_{\{0\}\times G}$ and $\cw(\va\times\psi)|_{G\times\{0\}}$.
\end{enumerate}
\end{Lemma}
\begin{proof}
\begin{enumerate}[(i)]
\item (see \cite{Da-Le}) If $F\in L^\infty(C_{\va\times\psi},\kappa_{\va\times\psi})$ satisfies $F\circ T_{\va\times\psi}=F$ and $F(x,g_1,g_2+g)=F(x,g_1,g_2)$ for each $g\in G$ then $F=F(x,g_1)$ and $F$ is constant by the ergodicity of $T_{\va}$.
\item Follows directly from~\eqref{warMa} applied to $\theta=\va\times \psi$.
\item This claim follows from (i) and Lemma~\ref{l2}.
\end{enumerate}
\end{proof}

\begin{Lemma} \label{l8}
Assume that $T,T'\in{\rm Aut}\xbm$ are  ergodic. Fix $\eta\in J^e(T,T')$. Let $\va\colon X\to G$ ($\va'\colon X\to G$) be an ergodic cocycle over $T$ (over $T'$) such that $(T\times T',\eta)_\va$~\footnote{$\va$ is treated as if it was defined on $(X\times X,\eta)$: $\va(x,y)=\va(x)$.} and $(T\times T',\eta)_{\va'}$ are also ergodic. Assume, moreover, that the Mackey action $\cw(\va\times\va')=\cw(\va\times\va',T\times T',\eta)$ associated to $\va\times\va'\colon X\times X\to G\times G$ is weakly mixing. Finally, let $\cs=(S_g)_{g\in G}\subset{\rm Aut}\ycn$ be an ergodic $G$-action. Then the  set
$$
J(T_{\va,\cs}, T'_{\va',\cs};\eta):=\{\eta'\in J(T_{\va,\cs},T'_{\va',\cs}) : \eta'|_{\cb\ot\cb}=\eta\}
$$
is a singleton (consisting of the relatively independent extension of $\eta$).
\end{Lemma}
\begin{proof}
Notice first that $J(T_{\va,\cs}, T'_{\va',\cs};\eta)$ is a subset  of $\cp((T\times T')_{\va\times \va',\cs\ot \cs},\cb\ot\cb,\eta)$.  Note that the coordinate $\sigma$-algebras in $(Y\times Y,\cc\ot\cc)$ (which are $\cs\ot \cs$-invariant) yield two $(T\times T')_{\va\times \va',\cs\ot \cs}$-invariant  $\sigma$-algebras $\cb\ot\cb\ot{\cal E}_i$, $i=1,2$. More precisely, $J(T_{\va,\cs}, T'_{\va',\cs};\eta)$ can be identified (cf.\ Remark~\ref{lele}) with
\begin{equation}\label{P1}
\mathcal{P}_1:=\{\eta'\in \cp((T\times T')_{\va\times \va',\cs\ot \cs},\cb\ot\cb,\eta): \eta'|_{\cb\ot\cb\ot{\cal E}_i}=\eta\ot\nu, i=1,2\}
\end{equation}
and it is enough to show that $\cp_1$ is a singleton. But, in view of Lemma~\ref{dlp}, the affine isomorphism
$$
A\colon\cp(T_{\va\times \va',\cs\ot \cs},\cb\ot\cb,\eta)\to \cp(\cw(\va\times\va')\times (\cs\otimes \cs) ,{\cal J}_{\va\times\va'},\kappa_{\va\times\va'})
$$
satisfies $A(\eta')|_{{\cal J}_{\va\times\va'}\ot{\cal E}_i}=\kappa_{\va\times\va'}\ot \nu$, $i=1,2$, provided that $\eta'\in\cp_1$. Moreover, it is not hard to see that the assumptions of Lemma~\ref{l6} are fulfilled (cf.\ Lemma~\ref{l3} (i)). The result follows.
\end{proof}

\subsection{Proof of Theorem~\ref{jm}}
We want to prove the AOP property of $\tfs$ and to obtain it, we need to show that (see~\eqref{aop2016})
$$
\limsup_{r\neq s,r,s\in\mathscr{P},r,s\to\infty} \sup_{\kappa\in J^e((\tfs)^r,(\tfs)^s)}\left|\int_{X\times Y\times X\times Y}f_1\ot g_1\ot f_2\ot g_2 \,d\kappa\right|=0$$
if $f_i\ot g_i\in L^2_0(X\times Y,\mu\ot\nu)$,  $i=1,2$. In fact, since $T$ is assumed to enjoy the AOP property, we can assume that
at least one of the functions $g_i$ is in $L^2_0\ycn$. But then the assumptions of Theorem~\ref{jm} and Lemma~\ref{l8} applied to $T^r$, $T^s$, $\va^{(r)}$ and $\va^{(s)}$   instead of $T$, $T'$, $\va$ and $\va'$ will tell us that the only members in $J^e((\tfs)^r,(\tfs)^s)$ are the relatively independent extensions of ergodic joinings in $J(T^r,T^s)$. The result follows from~\eqref{rieCALKA}. \qed

\subsection{A special case when $T$ is a totally ergodic rotation}\label{specialcase}
Let us consider now the special case of Theorem~\ref{jm} in which $T$ is a totally ergodic rotation. Then $T$ has the AOP property \cite{Ab-Le-Ru}.
With no loss of generality, we can assume that $X$ is a compact metric monothetic group, $Tx=x+\alpha$, where $\{n\alpha :  n\in \Z\}$ is dense in $X$. The measure $\mu$ is then $\la_X$ the Haar measure on $X$.

We first repeat the argument from \cite{Ku-Le} giving rise to the full description of ergodic joinings between $T^r$ and $T^s$ with $(r,s)=1$. For this aim, choose $a,b\in\Z$ so that $ar+bs=1$. Fix $u\in X$ and consider
$$A_u:=\{(x,y+u)\in X\times X :  sx=ry\}.$$
Then the map $V_u(x,y+u)=ax+by$ settles a topological isomorphism of the action of $T^r\times T^s$ on $A_u$ and of $T$ on $X$ with the latter action being uniquely ergodic. Hence, we described $J^e(T^r,T^s)$.
Moreover,
$$
X\times X \times G\times G= \bigcup_{u\in X} A_{u}\times G\times G
$$
is a decomposition of $X\times X \times G\times G$ into pairwise disjoint closed sets invariant under $(T^r\times T^s)_{\va^{(r)}\times \va^{(s)}}$. Moreover, $(T^r\times T^s)_{\va^{(r)}\times \va^{(s)}}|_{A_u\times G\times G}$ is topologically isomorphic to $T_{\va^{(r)}(r\cdot)\times \va^{(s)}(s\cdot+u)}$.
Indeed, the isomorphism is given by $J=J_u\colon A_u\times G\times G \to X\times G\times G$, $J(x,y+u,g,h)=(ax+by,g,h)$.
Finally notice that if $\eta\in J^e(T^r,T^s)$ then
$(T^r\times T^s,\eta)_{\va^{(r)}}$ has, via the new coordinates given by $J_u$, the form $T_{\va^{(r)}(r\cdot)}$. In view of Theorem~\ref{jm}, we hence obtain the following result.

\begin{Cor}\label{c:jm}
Assume that $T\in{\rm Aut}\xbm$ is a totally ergodic ergodic rotation on a compact metric monothetic group $X$. Let $G$ be an lcsc Abelian  group. Assume that
$\va\colon X\to G$ is a cocycle such that:
\begin{align}
&\va^{(r)}(r\,\cdot)\text{ is ergodic for each }r\in\cp;\label{E40}\\
&\text{the Mackey action }W(\va^{(r)}(r\,\cdot)\times \va^{(s)}(U(s\cdot))\text{ is weakly mixing}\label{E50}
\end{align}
for each $r\neq s$, $r,s\in\cp$ and arbitrary $U\in C(T)$.\footnote{We recall that the centralizer of an ergodic rotation on a compact metric group consists of all other rotations.} Let $\cs=(S_g)_{g\in G}\subset{\rm Aut}\ycn$ be an ergodic $G$-action. Then $\tfs$ has the AOP property.
\end{Cor}

\section{Smooth cocycles over a generic irrational rotation}
\subsection{Smooth Anzai skew products having the AOP property}
In this section our aim will be to prove the following.
\begin{Prop}\label{wn:zalozenia}
 Assume that $f\in C^{1+\delta}(\T)$, $\int_{\T}f\,d\la_{\T}=0$, for some $\delta>0$ and it is not a trigonometric polynomial. Then, there exists a dense $G_\delta$ set of  $\alpha$ such that, for $Tx=x+\alpha$, we have
\begin{align}
&\mbox{$f^{(r)}{(r\cdot)}$ is ergodic for each $r\in\mathscr{P}$;}\label{eq:jo1}\\
& \text{W($f^{(r)}(r\cdot)\times f^{(s)}(s\cdot+h)$ is weakly mixing for $r<s$, $r,s\in\mathscr{P}$, $h\in\T$.}\label{eq:jo2}
\end{align}
 \end{Prop}

As a consequence of Proposition~\ref{wn:zalozenia} and Corollary~\ref{c:jm}, we obtain the following.

\begin{Cor}\label{c:aopAnzai}
Under the assumptions of Proposition~\ref{wn:zalozenia}, for each ergodic flow $\cs=(S_t)_{t\in\R}\subset{\rm Aut}\ycn$, the automorphism $\tfs$ has the AOP property.
\end{Cor}

Combining Corollary~\ref{c:aopAnzai}  with Proposition~\ref{p:rwm}, we obtain the following.

\begin{Cor}\label{c:rwmAOP}
There are relatively weakly mixing extensions $\tfs$ of an irrational rotation $T$ which have the AOP property and are disjoint from all weakly mixing transformations.\end{Cor}

\subsection{Proof of Proposition~\ref{wn:zalozenia}}
Let $f\colon \R\to\R$ (periodic of period $1$) be in $L^2(\T)$, $f(x)=\sum_{n=-\infty}^{\infty}\widehat{f}(n)e^{2\pi i n x}$. Assume that $f$ has zero mean and its Fourier transform is absolutely summable. Let
\begin{multline*}
f_m(x):=f(x)+f\left(x+\frac{1}{m}\right)+\dots+f\left(x+\frac{m-1}{m}\right)\\
=m\sum_{l=-\infty}^{\infty}\widehat{f}(lm)e^{2\pi i lmx},\text{ for }m\geq 1.
\end{multline*}
Recall the following ergodicity criterion:
\begin{Th}[Theorem 5.1 in \cite{Aa-Le-Ma-Na}]\label{tw:aa}
Suppose that there exist a sequence $(q_n)\subset \N$ and a constant $C>0$ such that
\begin{itemize}
\item
$q_n\sum_{l=-\infty}^{\infty}|\widehat{f}(lq_n)|\leq C\|f_{q_n}\|_{L^2}$ for $n\geq 1$,
\item
$0<\|f_{q_n}\|_{L^2}\to 0$.
\end{itemize}
Then there exists a dense $G_\delta$ set of irrational numbers $\alpha$ such that the corresponding group extension $T_f\colon \T\times\R\to\T\times\R$, where $Tx=x+\alpha$, is ergodic.
\end{Th}

We will now prove a modified version of Theorem~\ref{tw:aa} (the proof follows the lines of the proof of Theorem 5.1 in~\cite{Aa-Le-Ma-Na}):
\begin{Th}\label{pr:katka1}
Under the assumptions of Theorem~\ref{tw:aa}, there exists a dense $G_\delta$ set of irrational numbers $\alpha$ such that the group extensions $T_{f^{(r)}(r\cdot)}$, where $Tx=x+\alpha$, are ergodic for all $r\geq 1$.
\end{Th}
We will need the following lemma:
\begin{Lemma}[\cite{Aa-Le-Ma-Na}]\label{twlm1}
Given $C>0$, there exist numbers $K,L,M$ such that $0<K<L<1$, $0<M<1$ and for each $h\in L^4(\T)$ if $\|h\|_{L^4}\leq C\|h\|_{L^2}$ then
$$
\mu\big(\{x\in\T\colon K\|h\|_{L^2}\leq |h(x)| \leq L\|h\|_{L^2}\}\big)>M.
$$
\end{Lemma}
Recall also (see e.g.~\cite{Kw-Le-Ru}) that given an infinite set $\{q_n\}_{n\in\N}\subset \N$ and a positive real valued function $R=R(q_n)$ the set
\begin{multline}\label{eq:rezydualny}
\mathcal{A}=\left\{\alpha\in [0,1) : \text{ for infinitely many }n\text{ we have }\left|\alpha-{p_n}/{q_n} \right|<R(q_n),\right.\\
\left.\text{where }{p_n}/{q_n}\text{ are convergents of }\alpha\right\}\text{ is residual.}
\end{multline}

\begin{proof}[Proof of Theorem~\ref{pr:katka1}]
For $n\geq 1$, let
$$
\widetilde{f}_n(x):=q_n\sum_{l=-\infty}^{\infty}\widehat{f}(lq_n)e^{2\pi i lx}.
$$
Fix $r\geq 1$ and let $g(x):=f^{(r)}(rx)$. Then for all $m$,
\begin{equation}\label{eq:gie}
g^{(m)}(x)=f^{(rm)}(rx).
\end{equation}
Moreover, in view of the assumptions of the theorem:
\begin{itemize}
\item
$\widetilde{f}_n(q_nx)=f_{q_n}(x)$, whence $\|\widetilde{f}_n\|_{L^2}=\|f_{q_n}\|_{L^2}\to 0$,
\item
$\|\widetilde{f}_n\|_{L^4}\leq \|\widetilde{f}_n\|_{L^\infty}\leq q_n\sum_{l=-\infty}^{\infty}|\widehat{f}(lq_n)|\leq C\|f_{q_n} \|_{L^2}=C\|\widetilde{f}_n\|_{L^2}$.
\end{itemize}
Therefore, by Lemma~\ref{twlm1}, for all $n\geq 1$,
$$
\mu\left(\left\{x\in\T\colon K\|\widetilde{f}_n\|_{L^2}\leq |\widetilde{f}_n(x)|\leq L\|\widetilde{f}_n\|_{L^2}\right\}\right)>M.
$$

Let $\{D_R\colon R\geq 1\}$ be a family of disjoint, closed intervals of the form $D_R=[c_R,d_R]$, where $\frac{d_R}{c_R}=100\frac{L}{K}$ and $0<d_R\to 0$. Let $(D'_n)_{n\geq 1}\subset \{D_R\colon R\geq 1\}$ be a sequence such that for all $R\geq 1$, we have $\# \mathcal{N}_R=\infty$, where
$$
\mathcal{N}_R=\{n\in\N\colon D_n'=D_R\}.
$$
Fix $n\in\N$. Choose $k_n,l_n\in\N$ so that
$$
\left[k_n K\|\widetilde{f}_{l_n}\|_{L^2},k_nL \|\widetilde{f}_{l_n}\|_{L^2}\right]\subset \widetilde{D}_n'\subsetneq D'_n,
$$
where $\widetilde{D}'_n$ is a strict closed subinterval of $D'_n$. This gives us two sequences $(k_n)_{n\geq 1}$ and $(l_n)_{n\geq 1}$, such that
$$
\mu\left(\left\{x\in\T\colon |k_n\widetilde{f}_{l_n}(x)|\in \widetilde{D}'_n\right\}\right)\geq M.
$$

We fix now $R\geq 1$. We claim that the set of numbers $\alpha$ such that there exists an infinite subset $\mathcal{N}^\alpha_R\subset \mathcal{N}_R$ satisfying the following two conditions:
\begin{enumerate}[(i)]
\item\label{EQ:i}
$\|k_nq_{l_n}\alpha\| \xrightarrow[n\in \mathcal{N}_R^\alpha]{}0$,
\item\label{EQ:ii}
for $n\in \mathcal{N}_R^\alpha$, we have
\begin{multline*}
\mu\left(\left\{y\in I\colon \left|f^{(k_nq_{l_n})}(T^{ik_nq_{l_n}}y)\right|\in D_n' \text{ and } \right.\right.\\
\left.\left.\text{sgn} (f^{(k_nq_{l_n})}(T^{ik_nq_{l_n}}y))\text{ is the same for }0\leq i<r\right\}\right)\geq M|I|
\end{multline*}
for every interval $I$ such that $|I|=\frac{t}{q_{l_n}}$, $t=1,\dots,q_{l_n}$,
\end{enumerate}
is a dense $G_\delta$ subset of $\T$. In view of~\eqref{eq:rezydualny}, it suffices to show that the above conditions describe a speed of approximation of $\alpha$ by rational numbers. This is clearly the case for~\eqref{EQ:i}. We will now deal with~\eqref{EQ:ii}. For $0\leq i<r$, we have
\begin{align*}
\left|f^{(kq)}(T^{ikq}y)-kf_q(y) \right|&\leq \sum_{j=0}^{k-1}\left|f^{(q)}(y+ikq\alpha+jq\alpha)-f_q(y) \right|\\
&=\sum_{j=0}^{k-1}\left|\sum_{w=0}^{q-1}f(y+ikq\alpha+jq\alpha+w\alpha) - \sum_{w'=0}^{q-1}f\left(y+\frac{w'}{q}\right) \right|\\
&=\sum_{j=0}^{k-1}\left|\sum_{w=0}^{q-1}\left(f(y+ikq\alpha+jq\alpha+w\alpha)-f\left(y+w\frac{p}{q}\right) \right) \right|\\
&\leq \sum_{j=0}^{k-1}\sum_{w=0}^{q-1}\omega\left(f, ikq\alpha+jq\alpha+w\left(\alpha-\frac{p}{q}\right)\right),
\end{align*}
where $p=p_{l_n}$, $q=q_{l_n}$ (recall that $p_{l_n}$ and $q_{l_n}$ are relatively prime) and $\omega(f,\cdot)$ is the modulus of continuity of $f$. With $k$ and $q$ fixed, the above quantity depends only on the distance between $\alpha$ and $\frac{p}{q}$.

We denote the obtained dense $G_\delta$ set of numbers $\alpha$ by $Y_R$ and take $Y=\bigcap_{R=1}^{\infty}Y_R$. This set is again a dense $G_\delta$. Pick $\alpha\in Y$ and fix $R\geq 1$. Let $J$ be an interval such that $|J|=\frac{t}{rq_{l_n}}$ for some $t=1,\dots,q_{l_n}$. Then, for $I:=rJ$, we have $|I|=\frac{t}{q_{l_n}}$, i.e.\ condition~\eqref{EQ:ii} holds. Therefore
\begin{equation}
\begin{split}\label{spli}
\mu\big(\big\{x\in J\colon &  \left|f^{(rk_nq_{l_n})}(rx) \right|\in rD'_n \big\}\big)\\
\geq & \mu\left(\left\{x\in J\colon \left| f^{(k_nq_{l_n})}(T^{ik_nq_{l_n}}rx) \right|\in D'_n\text{ and }\right.\right.\\
&\left.\left.\text{sgn} (f^{(k_nq_{l_n})}(T^{ik_nq_{l_n}}rx))\text{ is the same for }0\leq i<r \right\}\right)\\
=&\frac{1}{r}\cdot\mu\left(\left\{y\in I\colon \left| f^{(k_nq_{l_n})}(T^{ik_nq_{l_n}}y) \right|\in D'_n\text{ and }\right.\right.\\
&\left.\left.\text{sgn} (f^{(k_nq_{l_n})}(T^{ik_nq_{l_n}}y))\text{ is the same for }0\leq i<r \right\}\right)\\
\geq & \frac{1}{r}\cdot M |I|=M|J|.
\end{split}
\end{equation}

Suppose now that $E(g)=\lambda\Z$ for some $\lambda\in\R$. Let $R\geq 1$ be large enough, so that the set $K_R:=r(D_R\cup(-D_R))$ is disjoint from $\lambda\Z$.  By Proposition~\ref{p:disjess}, there exists a Borel set $B$ with $\mu(B)>0$ and such that for all $m\geq 1$, we have
$$
\mu\left(B\cap T^{-m}B\cap\left\{x\in \T\colon g^{(m)}(x)\in K_R\right\}\right)=0.
$$
Therefore, taking $m=k_nq_{l_n}$, in view of~\eqref{eq:gie}, we obtain
\begin{equation}\label{con}
\mu\left( B\cap T^{-k_nq_{l_n}}B\cap \left\{x\in\T\colon f^{(rk_nq_{l_n})}(rx)\in K_R\right\}\right)=0.
\end{equation}
If now $n\in \mathcal{N}_R^\alpha$ then $\mu(B\triangle T^{k_nq_{l_n}}B)\to 0$ by condition~\eqref{EQ:i} ($k_nq_{l_n}$ is a rigidity sequence for $T$). If $x$ is a density point of $B$ then for an interval $J$, containing $x$, with $|J|=\frac{t}{rq_{l_n}}$, for some $t=1,\dots,q_{l_n}$, we have $\mu(B\cap J)>(1-\frac{M}{2})|J|$. In view of~\eqref{spli}, it follows that there exists a measurable subset $A_n\subset B$ with $\mu(A_n)\geq \frac{M}{2}\mu(B)$ such that
$$
\left|f^{(rk_nq_{l_n})}(rx) \right|\in rD'_n \text{ for }x\in A_n.
$$
Let $\widetilde{A}_n:=A_n\cap T^{-k_nq_{l_n}}B$. Then, by condition (i), for $n$ large enough, $\mu(\widetilde{A}_n)\geq \frac{M}{4}\mu(B)$. Hence
$$
\mu\left(B\cap T^{-k_nq_{l_n}}B\cap \{x\in \T\colon f^{(rk_nq_{l_n})}(rx)\in K_R\} \right)\geq\mu(\widetilde{A}_n)\geq\frac{M}{4}\mu(B).
$$
This contradicts~\eqref{con} and completes the proof.
\end{proof}

\begin{Remark}\label{uw:brakujaca}
It was shown in~\cite{Ku-Le} that the assumptions of Theorem~\ref{tw:aa} are satisfied for each zero mean function $f\in C^{1+\delta}(\T)$, $\delta>0$, which is not a trigonometric polynomial.
\end{Remark}

\begin{Th}[Cor.\ 2.5.6 in \cite{Ku-Le}]\label{tw:slabemieszanie} Assume that $f\in C^{1+\delta}(\T)$ for some $\delta>0$ and it is not a trigonometric polynomial. Then, for a dense $G_\delta$ set of $\alpha\in\T$, we have: for arbitrary $A,B\in\R$, $|A|+|B|>0$, any relatively prime numbers $r,s\in\N$, any $|c|=1$ and $h\in\R$, the cocycle
$$
ce^{2\pi i(Af^{(r)}(r\cdot)+Bf^{(s)}(s\cdot+h))}$$
considered over $Tx=x+\alpha$, is not a $T$-coboundary.\end{Th}

Proposition~\ref{wn:zalozenia} follows immediately by Theorem~\ref{pr:katka1}, Theorem~\ref{tw:slabemieszanie} and Lemma~\ref{l3}~(ii).

\section{Applications -- average orthogonality on short intervals}
\subsection{Universal sequences $(a_n)$ for the strong MOMO property along $(a_n)$}
Assume that $T$ is a uniquely ergodic homeomorphism of a compact metric space $X$ with the unique $T$-invariant measure $\mu$ (hence $T\in{\rm Aut}\xbm$ is ergodic). Assume that $G$ is an lcsc Abelian  group and let $\va\colon X\to G$ be continuous. Assume moreover that
$\cs=(S_g)_{g\in G}$ is a continuous\footnote{I.e.\ the map $(g,x)\to S_gx$ is continuous.} uniquely ergodic $G$-action on a compact metric space $Y$ (with the unique invariant measure $\nu$). It follows that $\tfs$ is a homeomorphism of $X\times Y$. According to Lemma~\ref{dlp}, the simplex of invariant measures for $\tfs$ is affinely isomorphic to $\cp(\cw(\va)\times\cs,\cj_\va,\kappa_\va)$. It follows that if $\cw(\va)$ is trivial then $\tfs$ will be uniquely ergodic. We hence proved the following.

\begin{Prop}[cf.\ \cite{Le-Le}] \label{p:ueRE} If $T$, $\cs$ are uniquely ergodic, $\va\colon X\to G$ is continuous and ergodic then
$\tfs$ is a uniquely ergodic homeomorphism of $X\times Y$.
\end{Prop}

Using Proposition~\ref{AOPtoMOMO}, we will now obtain the existence of universal sequences $(a_n)\subset G$ for which the strong MOMO property along $(a_n)$ holds for all continuous uniquely ergodic $G$-actions. Such sequences will be ergodic by Proposition~\ref{ergseq}.

\begin{Th}\label{t:universalMOMO} Assume that $T$ is uniquely ergodic, $\va\colon X\to G$ is continuous and ergodic, and the other assumptions in Theorem~\ref{jm} are also satisfied. Let $x\in X$ and set $a_n:=\va^{(n)}(x)$, $n\geq0$. Then each  
continuous uniquely ergodic $G$-action $\cs=(S_g)_{g\in G}$ on a compact metric space $Y$ satisfies the  strong MOMO property along $(a_n)$.

In particular, each sequence $(f(S_{a_n}x))$ (with $f\in C(Y)\cap L^2_0\ycn$) is orthogonal to an arbitrary multiplicative function $\bfu$, $|\bfu|\leq1$.
\end{Th}

\subsection{From universal sequences for flows to universal sequences for automorphisms}

So far the only cocycles satisfying the assumptions of Theorem~\ref{jm} (and the more the assumptions of Theorem~\ref{t:universalMOMO}) are smooth cocycles over some irrational rotations. We will now show how to use them to obtain integer-valued sequences universal for the strong MOMO property for uniquely ergodic homeomorphisms.

\begin{Remark}\label{r:ueSF} Note that if $R$ is a uniquely ergodic homeomorphism of a compact metric space $Z$, then the suspension $\widetilde{R}$ is  a continuous flow on the compact metric space $\widetilde{Z}$ and the flow is also uniquely ergodic.\footnote{The metric on $Z\times[0,1]/\sim$ is the quotient metric of the natural product metric on $Z\times [0,1]$. With this metric, the map $((x,s),t)\mapsto (R^{[t]}x,\{s+t\})$ becomes continuous.}\end{Remark}

Let $T$ be a uniquely ergodic homeomorphism of a compact metric space $X$ and let $\va\colon X\to\R$ be continuous and ergodic. Fix $x\in X$, set $a_n=\va^{(n)}(x)$. Assume that we have the strong MOMO property along $(a_n)$ for the suspension flow $\widetilde{R}$. If we fix $f\in C(X)$, $\int f\,d\mu=0$, and some $\delta>0$ then we can find $F\in C(\widetilde{Z})$, $\int F\,d(\mu\ot\la_{[0,1]})=0$, $\|F\|_\infty=\|f\|_\infty$,
and
\beq\label{tasama}
\mbox{$F(z,s)=f(z)$ whenever $\delta\leq s\leq1-\delta$.}
\eeq
Indeed, it is enough to set
$$F(z,s)=\left\{\begin{array}{lll}
\frac{s}\delta f(z)&\text{if} &0\leq s\leq\delta\\
f(z)& \text{if}& \delta\leq s\leq1-\delta\\
\frac1\delta(1-s)f(z)& \text{if}& 1-\delta\leq s\leq1.\end{array}\right.
$$
Choose $(z_k)\subset Z$ arbitrarily.
Now, we have
$$
\frac1{b_{K+1}}\sum_{k\leq K}\left|\sum_{b_k\leq n<b_{k+1}}F(\widetilde{R}_{a_n}(z_k,0))\bfu(n)\right|\to0.$$
That is
\beq\label{dozera}
\frac1{b_{K+1}}\sum_{k\leq K}\left|\sum_{b_k\leq n<b_{k+1}}F(R^{[a_n]}(z_k,\{a_n\}))\bfu(n)\right|\to0.\eeq
However, the sequence $(a_n)$ is uniformly distributed mod~1 (see Remark~\ref{r:ude}), hence
$$
\frac1{b_{K+1}}|\{0\leq n< b_{K+1} :  \{a_n\}\in[0,\delta]\cup[1-\delta,1)\}|<3\delta$$
if $K$ is large enough. Hence, taking into account additionally~\eqref{tasama} and~\eqref{dozera}, we obtain that
\beq\label{dozeraA}
\frac1{b_{K+1}}\sum_{k\leq K}\left|\sum_{b_k\leq n<b_{k+1}}f(R^{[a_n]}z_k)\bfu(n)\right|\to0.\eeq

We have proved the following.

\begin{Cor}\label{c:momoA}
Assume that $T$ is uniquely ergodic, $\va\colon X\to \R$ is continuous and ergodic, and the other assumptions in Theorem~\ref{jm} are also satisfied. Let $x\in X$ and set $a_n=\va^{(n)}(x)$, $n\geq0$. Then  each uniquely ergodic homeomorphism $R$ of a compact metric space $Z$  satisfies the strong  MOMO property along $([a_n])$.

In particular, each sequence $(f(R^{[a_n]}z))$ (with $f\in C(Z)\cap L^2_0\zdr$) is orthogonal to an arbitrary multiplicative function $\bfu$, $|\bfu|\leq1$.\end{Cor}

\section{Affine cocycles over irrational rotations. AOP and the strong MOMO property}\label{s:AOPaf}
We continue to study the AOP property for automorphisms of the form $\tfs$, where $T$ is an irrational rotation. From now on, we assume that $\va(x)=x-\frac12=\{x\}-\frac12$. This map is considered as a cocycle over an irrational rotation by $\alpha$. (Our analysis is true for a general (zero mean) affine cocycle $\va(x)=mx+c$ with $0\neq m\in\Z$  but for simplicity of notation we will only consider the case $m=1$.)

The cocycle $\va\colon X\to\R$ is not continuous, so the first task will be to show how to bypass this inconvenience. In fact, $\va$ is a continuous cocycle by a slight extension of the base -- in Subsection~\ref{sturmianmodel} we will show that studying $\va$ over this extension still yields the results we are interested in.

\subsection{Lifting generic points in the Cartesian square}\label{sturmianmodel}
Assume that $T$ and $\widehat{T}$ are uniquely ergodic homeomorphisms of compact metric spaces $X$ and $\widehat{X}$, with the unique invariant measures $\mu$ and $\widehat{\mu}$, respectively. Assume moreover, that $\pi\colon \widehat{X}\to X$ is continuous and $T\circ \pi = \pi \circ \widehat{T}$, in particular, $(X,T)$ is a topological factor of $(\widehat{X},\widehat{T})$.

\begin{Prop}\label{p:d11}
Assume that $(\widehat{X},\widehat{\mu},\widehat{T})$ and $(X,\mu,T)$ are measure-theoretically isomorphic. Assume that $T$ is measure-theoretically coalescent.\footnote{$T$ is called coalescent if each measure-preserving map $W$ on $(X,\mu)$ commuting with $T$ is invertible. Equivalently, for each factor $\ca\subset\cb$ if $T|_{\ca}$ is isomorphic to $T$ then $\ca=\cb$.} Finally, assume that each pair $(x,y)\in X\times X$ is generic for an ergodic $T\times T$-invariant measure. Then each pair $(\widehat{x},\widehat{y})\in \widehat{X}\times\widehat{X}$ is generic for some $\widehat{T}\times\widehat{T}$-invariant measure. Moreover, if $(\widehat{x},\widehat{y})\in \widehat{X}\times\widehat{X}$ is generic for $\widehat{\rho}$ then $(\widehat{X}\times\widehat{X},\widehat{\rho}, \widehat{T}\times\widehat{T})$ is isomorphic to $(X\times X,\rho,T\times T)$, where $\rho=(\pi\times\pi)_\ast (\widehat{\rho})$.
\end{Prop}
\begin{proof}
Take $(\widehat{x},\widehat{y})\in \widehat{X}\times \widehat{X}$ and let $(x,y):=(\pi(\widehat{x}),\pi(\widehat{y}))$. Assume that
\begin{equation}\label{e1}
\frac1{N_k}\sum_{n\leq N_k}\delta_{(\widehat{T}\times \widehat{T})^n(\widehat{x},\widehat{y})}\to \widehat{\rho},\text{ when }k\to\infty.
\end{equation}
Since $(x,y)$ is generic, for some measure $\rho$, we have
$$
\frac1N\sum_{n\leq N}\delta_{(T\times T)^n(x,y)}\to \rho, \text{ when }N\to\infty.
$$
By the continuity of $\pi$, for $f,g\in C(X)$, we obtain
\begin{align*}
\int_{X\times X}f\ot g\ d((\pi\times\pi)_\ast(\widehat{\rho}))&=\int_{\widehat{X}\times \widehat{X}}(f\circ\pi)\ot(g\circ\pi)\,d\,\widehat{\rho}\\
&=\lim_{k\to\infty}\frac1{N_k}\sum_{n\leq N_k}f\circ \pi(\widehat{T}^n\widehat{x})\cdot g\circ \pi(\widehat{T}^n\widehat{y})\\
&=\lim_{k\to\infty}\frac1{N_k}\sum_{n\leq N_k}f(T^nx)\cdot g(T^ny)=\int_{X\times X}f\ot g\,d\rho.
\end{align*}
Hence $\rho=(\pi\times \pi)_\ast(\widehat{\rho})$.
By coalescence, $\pi^{-1}(\cb(X))=\cb(\widehat{X})$ (modulo $\widehat{\mu}$), so there exists a $T$-invariant set $\widehat{X}_0\subset \widehat{X}$ such that $\widehat{\mu}(\widehat{X}_0)=1$ and $\pi|_{\widehat{X}_0}$ is 1-1.
It follows that
$$
\widehat{\rho}(\widehat{X}_0\times \widehat{X}_0)=\widehat{\rho}((\widehat{X}\times \widehat{X}_0)\cap(\widehat{X}_0\times \widehat{X}))=1
$$
since $\widehat{\rho}(\widehat{X}\times \widehat{X}_0)=\widehat{\rho}(\widehat{X}_0\times \widehat{X})=\widehat{\mu}(\widehat{X}_0)=1$ ($(\widehat{X},\widehat{T})$ is uniquely ergodic, so the projections of $\widehat{\rho}$ on both coordinates are equal $\widehat{\mu}$). Since  $\pi$ is 1-1 on $\widehat{X}_0$,
$\pi\times \pi$ is 1-1 on $\widehat{X}_0\times \widehat{X}_0$.
Moreover, for Borel sets $\widehat{A},\widehat{B}\subset \widehat{X}$, we have
\begin{multline*}
\widehat{\rho}(\widehat{A}\times \widehat{B})=\widehat{\rho}((\widehat{A}\cap \widehat{X}_0)\times (\widehat{B}\cap \widehat{X}_0))\\
=\widehat{\rho}(\pi^{-1}(\pi(\widehat{A}\cap \widehat{X}_0))\times \pi^{-1}(\pi(\widehat{B}\cap \widehat{X}_0)))=\rho(\pi(\widehat{A}\cap \widehat{X}_0)\times \pi(\widehat{B}\cap \widehat{X}_0))
\end{multline*}
As $\widehat{X}_0$ depends only on $\widehat{\mu}$, it follows that $\widehat{\rho}$ does not depend on the choice of $(N_k)$ in~\eqref{e1}, i.e.\ $(\widehat{x},\widehat{y})$ is generic. Moreover, $(X\times X,\rho, T\times T)$ is isomorphic to $(\widehat{X}\times\widehat{X},\widehat{\rho}, \widehat{T}\times\widehat{T})$. In particular, $\widehat{\rho}$ is ergodic.
\end{proof}
\begin{Remark}
Proposition~\ref{p:d11} remains true for Cartesian products of two different systems (the proof is the same). More precisely, we consider uniquely ergodic extensions $\widehat T_i$ of $T_i$, $i=1,2$, where both $T_1$ and $T_2$ are coalescent and isomorphic to $\widehat T_1$ and $\widehat T_2$, respectively. In particular, we can apply this to powers $T^r,T^s$ knowing that each each measure-preserving map commuting with $T^k$ commutes with $T$,  for each $k\in\Z\setminus\{0\}$.
\end{Remark}

Coming back to our affine situation, we consider $X=\T$, $Tx=x+\alpha$, an irrational rotation.
Then we make $\varphi\colon\T\to\R$, $\varphi(x)=x-\frac12$ continuous by considering $\varphi$ in the coordinates given by a Sturmian system $\widehat X\subset\{0,1\}^{\Z}$ considered with the shift $\widehat T$. More precisely, $\widehat X$ is given by the closure of the names of $0$ for the partition $[0,1/2)$, $[1/2,1)$ of the circle, see \cite{Arnoux} for details. Now, the Sturmian system
$(\widehat{X},\widehat T)$ is uniquely ergodic (with the unique invariant measure $\widehat{\mu}$).  Set $\widehat{\va}(\widehat{x}):=\va(\pi(\widehat{x}))$ which is a continuous cocycle on $\widehat{X}$. Note that:
\begin{itemize}
\item the map $(\widehat{x},y)\mapsto (\pi(\widehat{x}),y)$  settles a natural factor map from  $\widehat T_{\widehat\va,\cs}$ to $\tfs$, for each $\cs$,
\item      $\widehat T_{\widehat\va,\cs}$ is uniquely ergodic whenever $\cs$ is uniquely ergodic (in view of Proposition~\ref{p:uniqueE}), and the map $(\widehat{x},y)\mapsto (\pi(\widehat{x}),y)$ becomes a measure-theoretic isomorphism.
\end{itemize}

Notice that
\beq\label{sumy} \widehat{\va}^{(n)}(\widehat{x})=\va^{(n)}(\pi(\widehat{x})).\eeq

Using Proposition~\ref{p:d11}, we can also see that if $(x,x)$ is a generic point for an ergodic  $T^r\times T^s$-invariant measure $\rho$ then so is $(\widehat x,\widehat x)$ for $\widehat T^r\times\widehat T^s$ ($\pi(\widehat x)=x$). Moreover,  if $((x,y),(x,y))$ is generic for $\left(\tfs\right)^r\times\left(\tfs\right)^s$ for a unique measure $\widetilde{\rho}$
satisfying $\widetilde{\rho}|_{X\times X}=\rho$, also
$((\widehat x,y),(\widehat x,y))$ is generic for $\left(\widehat T_{\widehat\va,\cs}\right)^r\times \left(\widehat T_{\widehat\va,\cs}\right)^s$.

\subsection{Ergodicity of $\va^{(r)}(r\cdot)$}
We assume that $\va(x)=\{x\}-\frac12$. Then for each $r\geq1$, we have
\begin{align*}
\va^{(r)}(rx)&=\sum_{j=0}^{r-1}\left(\{rx+j\alpha\}-\frac12\right)\\
&=\sum_{j=0}^{r-1}\left(rx+j\alpha-[rx-j\alpha]-\frac12\right)=r^2x-\sum_{j=0}^{r-1}[rx+j\alpha]+\frac{r(r-1)}2\alpha.
\end{align*}
It follows that
\beq\label{er1}
\va^{(r)}(rx)=\Theta_r(x)-\Psi_r(x),\eeq
where
\beq\label{er2}
\Theta_rx=r^2x-\frac{r^2}2,\;\Psi_r(x)=\sum_{j=1}^{r-1}[rx+j\alpha]-
\frac{r(r-1)}2(\alpha+1).\eeq
Note that the first cocycle is real valued, while the second one takes values in $\frac{r(r-1)}2\alpha+\Z$. Each of them is of bounded variation and has zero mean.

\begin{Remark} \label{rozwazania} Our main idea is now to consider $\Theta_r^{(q_n)}$ and $\Psi_r^{(q_n)}$, $n\geq1$, where $(q_n)$ is the sequence of denominators of $\alpha$. Then, by the Denjoy-Koksma inequality, the distributions of the above random variables are contained in a bounded subset of $\R$. It is well known, e.g.\ \cite{Le-Me-Na,Le-Pa-Vo}, that any limit point of $\left(\Theta_r^{(q_n)}\right)_\ast$ is an absolutely continuous measure. The values of $\Psi_r^{(q_n)}$, by passing to a subsequence if necessary, are taken in a fixed finite set. In other words, we are in the situation of Lemma~\ref{suma}.
\end{Remark}

Consider the set $\{\{j\alpha\} : j=0,1,\ldots,r-1\}$. Then, order this set so that:
$$
1-\{j_0\alpha\}<1-\{j_2\alpha\}<\ldots<1-\{j_{r-1}\alpha\}.$$
It follows that $\Psi_r$ is a step cocycle and the discontinuities of $\Psi_r$ are  the points
$$
\frac{k-\{j_0\alpha\}}r<\frac{k-\{j_1\alpha\}}r<\ldots<
\frac{k-\{j_{r-1}\alpha\}}r,\;k=1,\ldots,r.$$
At each of the above point the jump is equal to~1 or to $-(r-1)$.

\begin{Remark}\label{rozwazania1} Assume that $\alpha$ has bounded partial quotients. The crucial observation is that in this case the
points $\frac{k-\{j_\ell\alpha\}}r$, $k=1,\ldots,r$, $\ell=0,1,\ldots,r-1$ are badly approximated by $\alpha$, e.g.\ \cite{Co-Pi,Fr-Le-Le}. In other words, the intervals whose endpoints are the consecutive discontinuities of the cocycle $\Psi^{(q_n)}_r$ are all of length ``comparable'' with $\frac1{q_n}$. It easily follows that if $\kappa$ is a limit point of the set $\{\left(\Psi_r^{(q_n)}\right)_\ast : n\geq1\}$, then it has an atom at a  (even non-zero) member of the set~$F$ (see Remark~\ref{rozwazania}). \end{Remark}

By Remark~\ref{rozwazania} and Remark~\ref{rozwazania1}, we obtain that the topological support
of a limit point of  $\left\{\left(\Theta_r^{(q_n)}-\Psi_r^{(q_n)}\right)_\ast : n\geq1\right\}$ contains a translation of an uncountable subset of $\R$.
In this way, we have proved the following.

\begin{Prop}\label{p:ergRaff}
If $\alpha$ has bounded partial quotients, then $\va^{(r)}(r\cdot)$ is ergodic for each $r\geq1$.\end{Prop}

It follows immediately that all cocycles of the form $\va^{(s)}(s\cdot+c)$ are also ergodic.

\subsection{Regularity of $(\va^{(r)}(r\cdot),\va^{(s)}(s\cdot))$}
Our aim will be now to study the cocycles
\beq\label{cocRS}
(\va^{(r)}(r\cdot),\va^{(s)}(s\cdot))=(\Theta_r,\Theta_s)
-(\Psi_r,\Psi_s)\eeq taking values in $\R^2$, $r,s\in\N$.
We will constantly assume that $\alpha$ has bounded partial quotients.

Note that $s^2x-\frac{s^2}2=\frac{s^2}{r^2}(r^2x-\frac{r^2}2)$, so $(\Theta_r,\Theta_s)(\T)\subset P_{r,s}$, where
$$
P_{r,s}:=\left\{\left(t,\frac{s^2}{r^2}t\right) : t\in\R\right\}
$$
is a closed subgroup of $\R^2$.
By examining the sequence $\left((\Theta_r,\Theta_s)^{(q_n)}\right)_\ast$, $n\geq1$, it is not hard to see that each limit point of such distributions is an absolutely continuous measure on the line $P_{r,s}$. It easily follows that $E((\Theta_r,\Theta_s))=P_{r,s}$. In order to study $(\Psi_r,\Psi_s)$, we need first the following observation.

\begin{Lemma}\label{rozklaRS} Assume that $(r,s)=1$ Then the non-zero discontinuity points:
$$
\frac{k-\{j_\ell\alpha\}}r, \frac{k'-\{j'_\ell\alpha\}}s,\;\ell=0,1,\ldots, r-1,$$
for $\Psi_r^{(q_n)}$ and $\Psi_s^{(q_n)}$, respectively, are pairwise different.
\end{Lemma}
\begin{proof}
Suppose that $\frac{m-\{j\alpha\}}r=\frac{u-\{k\alpha\}}s$ with $m,u\in\Z$. It follows that
$$
s(j\alpha-[j\alpha])-r(k\alpha-[k\alpha])=sm-ru,$$
whence $sj\alpha-rk\alpha\in\Z$. If $j\neq 0$ or $k\neq0$, this is possible only if $sj=rk$. But $(r,s)=1$ implies $s\divides k$ and $r\divides j$ and we seek solution for
$0\leq j<r$ and $0\leq k<s$, a contradiction. Similar argument works  in case $j=k=0$ and either $m\neq r$ or $u\neq s$. \end{proof}

It now again will follow from the analysis of limit points of the distributions $\left((\Psi_r,\Psi_s)^{(q_n)}\right)_\ast$, $n\geq1$, that there is a non-zero point in the topological support of a limit of such distributions. Taking into account the form of closed subgroups of $\R^2$, Lemma~\ref{suma} and setting
\beq\label{grupaHrs}
\ch_{r,s}:=E((\va^{(r)}(r\cdot),\va^{(s)}(s\cdot)))\eeq
we obtain the following.

\begin{Prop}\label{p:regularityRS}
Assume that $\alpha$ has bounded partial quotients and $(r,s)=1$.
Then there exists $0\neq d\in\R$ such that
$$
\ch_{r,s}=P_{r,s}+(d\Z,0).$$
In particular, $(\va^{(r)}(r\cdot),\va^{(s)}(s\cdot))$ is regular.
\end{Prop}
\begin{proof} We have already shown that for some $d'\neq0$, the group $P_{r,s}+(d'\Z,0)$ is contained in the group of essential values of the cocycle under consideration. Hence, we only need to notice that $(\va^{(r)}(r\cdot),\va^{(s)}(s\cdot))$ is not ergodic. Indeed, let $\widetilde{\chi}(u,v)=e^{2\pi i(2s^2u-2r^2v)}$. It is not hard to see that
\beq\label{nieerg}
\raz\neq\widetilde\chi\in \Lambda_{(\va^{(r)}(r\cdot),\va^{(s)}(s\cdot))}.\eeq
The result follows.
\end{proof}

\begin{Prop}\label{p:WLanih}
Under the above assumptions on $\alpha$, $r$ and $s$, we have
$$
\ch_{r,s}^\perp\subset \Q\times \Q.$$
\end{Prop}\begin{proof}
In view of Lemma~\ref{l:wanih}, it follows that (cf.~\eqref{nieerg} in the proof of Proposition~\ref{p:regularityRS})
\beq\label{dzi1}
\widetilde\chi\in\ch_{r,s}^\perp.
\eeq
Now,
$$\ch_{r,s}^\perp=\left(P_{r,s}+(d\Z,0)\right)^\perp\subset P_{r,s}^\perp=
\left\{e^{2\pi i(u,-\frac{r^2}{s^2}u)(\cdot,\cdot)} :  u\in\R\right\}.$$
By \eqref{dzi1}, $\widetilde{\chi}(d,0)=1$, whence
$e^{2\pi i 2s^2d}=1$ and hence $d\in\Q$. Now, $(d,0)\in \ch_{r,s}$ and if
$$
e^{2\pi i(du-\frac{r^2}{s^2}u\cdot 0)}=1$$
then $ud\in\Z$ and it follows that $u\in\Q$. Hence $(u,-\frac{r^2}{s^2}u)\in\Q\times\Q$ and the result follows.
\end{proof}

\subsection{Regularity of $(\va^{(r)}(r\cdot),\va^{(s)}(s\cdot+c))$}
\begin{Prop}\label{p:erg+reg}
Assume that $\alpha$ has bounded partial quotients and $(r,s)=1$.
Then:
\begin{enumerate}[(i)]
\item
if $c\notin \Q\alpha+\Q$ then the cocycle $(\va^{(r)}(r\cdot),\va^{(s)}(s\cdot+c))$ is ergodic,
\item
if $c\in \Q\alpha+\Q$ then the cocycle $(\va^{(r)}(r\cdot),\va^{(s)}(s\cdot+c))$ is  regular, its group
of essential values $\ch_{r,s,c}$ is cocompact of the form $P_{r,s}+(d\Z,0)$ and $\ch_{r,s,c}^\perp\subset\Q\times\Q$.
\end{enumerate}
\end{Prop}
\begin{proof} Let us prove (i) first. For this aim, notice that
\beq\label{dzi2}
\va^{(s)}(sx+c)=\Theta_s(x)-(\Psi_s(x+c)-sc).\eeq
As before, we would like to study the limit points of distributions
$$\left((\Theta_r,\Theta_s)^{(q_n)}-
(\Psi_r,\Psi_s(\cdot+c/s))^{(q_n)}-(0,q_nsc)\right)_\ast,\;n\geq1.$$
However, we can also consider distributions on the circle, namely
$$
\left(e^{2\pi i((a,b)\cdot((\Theta_r,\Theta_s)^{(q_n)}-
(\Psi_r,\Psi(\cdot+c/s))^{(q_n)}+(0,q_nsc)))}\right)_\ast,$$
with $a,b\in\Z\setminus\{0\}$. Then
$$e^{2\pi i((a,b)\cdot((\Theta_r,\Theta_s)^{(q_n)}-
(\Psi_r,\Psi_s(\cdot+c))^{(q_n)}+(0,q_nsc)))}=
e^{2\pi i((a,b)\cdot((\Theta_r,\Theta_s)^{(q_n)}+
(0,q_nsc)))}.$$
Now, since $\alpha$ has bounded partial quotients and $bsc\notin\Q\alpha+\Q$, the sequence $(e^{2\pi iq_nbsc})$, $n\geq1$, has infinitely many accumulation points. If $d\in\T$ is any of these accumulation points, this means that there is a line $P_{r,s}+(dm,0)$ with some $m\neq\Z\setminus\{0\}$ meeting $\ch_{r,s,c}$ (cf.\ the proof of Lemma~3 in \cite{Le-Me-Na}). Since we have infinitely many $d$ at our disposal, $\ch_{r,s,c}=\R\times\R$.

The proof of (ii) is much the same as the proof in case $c=0$.
\end{proof}

\subsection{AOP property for Rokhlin extensions given by affine cocycles}
Using Corollary~\ref{c:jedynosc}, Proposition~\ref{p:erg+reg} and the fact that the group of eigenvalues of $\cs\ot\cs$ is the Cartesian square of the group of eigenvalues of $\cs$, we obtain the following result.

\begin{Th}\label{t:aopAFF}
Assume that $Tx=x+\alpha$ with $\alpha$ irrational of bounded partial quotients. Let $\va(x)=x-\frac12$. Let $\cs$ be any ergodic
flow acting on a  probability standard Borel space $\ycn$.
If the group of eigenvalues of $\cs$ on $L^2_0\ycn$ does not meet $\Q$ then $\tfs$ has the AOP property.
\end{Th}

By replacing the base rotation $T$ by its Sturmian model $\widetilde{T}$, we can view $\va(x)=x-\frac12$ as a continuous function (see Subsection~\ref{sturmianmodel}). Then for each ergodic $\cs=(S_t)$ which has no non-trivial rational eigenvalues, we obtain $\tfs$ uniquely ergodic, enjoying the strong MOMO property.

\begin{Cor}\label{c:affine1} For each  $(y_k)\subset Y$, $f\in C(Y)$ and each sequence $(b_k)\subset\N$ with $b_{k+1}-b_k\to\infty$, we have
$$
\frac1{b_{K+1}}\sum_{k\leq K}\left|
\sum_{b_k\leq n<b_{k+1}}f(S_{\va^{(n)}(0)}y_k)\bfu(n)\right|\to0, \text{ when}\;K\to\infty,
$$
for each multiplicative $\bfu\colon\N\to\C$, $|\bfu|\leq1$. Here, $$\va^{(n)}(0)=\frac{n(n-1)}2\alpha-\frac n2-\sum_{j=1}^{n-1}[j\alpha].
$$
\end{Cor}

Assume that $R$ is an ergodic automorphism of $\zdr$. Consider its suspension $\tilde{R}=(\tilde{R}_t)_{t\in\R}$, and change time in it:
$$
\ov{R}_t:=\widetilde R_{\beta t}.$$
The (additive) group of eigenvalues of $\ov{R}$ is the multiplication by $\beta$ of the group of eigenvalues of the suspension  (for the suspension the group of eigenvalues is the group for $R$ meant as a subgroup of $[0,1)$ and then we add $\Z$ to obtain an additive subgroup of $\R$).  We can choose $\beta$ so that the group of eigenvalues of $\ov R$ is hence disjoint from $\Q$ (mod the common element $0$).

By considering suspensions (and their change of time as above), we obtain:

\begin{Cor}\label{c:affine2} If $R$ is an ergodic automorphism of $\zdr$ and $\beta\in\R$ is such that
$\beta e(R)\subset\Q^c\cup\{0\}$ then
for each  $(z_k)\subset Z$, $f\in C(Z)$ and each sequence $(b_k)\subset\N$ with $b_{k+1}-b_k\to\infty$, we have
$$
\frac1{b_{K+1}}\sum_{k\leq K}\left|
\sum_{b_k\leq n<b_{k+1}}f(R^{[\beta\va^{(n)}(0)]}y_k)\bfu(n)\right|\to0, \;\text{when }K\to\infty,
$$
for each multiplicative $\bfu\colon\N\to\C$, $|\bfu|\leq1$.
\end{Cor}

Finally, by taking suspensions over the rotation by~1 on $\Z/\ell\Z$, we obtain the following.

\begin{Cor}\label{c:affine3} Let $\beta\in\Q^c$ be such that
$\beta e(R)\subset\Q^c\cup\{0\}$. Then,
for each sequence $(b_k)\subset\N$ with $b_{k+1}-b_k\to\infty$, we have
$$
\frac1{b_{K+1}}\sum_{k\leq K}\left|
\sum_{b_k\leq n<b_{k+1}}e^{2\pi i[\beta\va^{(n)}(0)]/\ell}\bfu(n)\right|\to0, \;\text{when }K\to\infty,
$$
for each multiplicative $\bfu\colon\N\to\C$, $|\bfu|\leq1$.
\end{Cor}

\subsection{The strong MOMO property for affine case}
In Subsection~\ref{sturmianmodel} we have considered $\tfs$ as a uniquely ergodic
automorphism (whenever $\cs$ is a uniquely ergodic flow). However, we have proved the AOP property of $\tfs$ only for those flows $\cs$ that have no non-trivial rational spectrum.

Consider a sequence of the form $(x,y_k)\subset X\times Y$.  We have
$$
\frac1{b_{K+1}}\sum_{k\leq K}\left(
\sum_{b_k\leq n<b_{k+1}}\delta_{\left(\tfs\right)^{rn}\times \left(\tfs\right)^{sn}((x,y_k),(x,y_k))}\right)\to\widetilde{\rho},$$
when $K\to\infty$. It follows from the above that
$$
\frac1{b_{K+1}}\sum_{k\leq K}\left(
\sum_{b_k\leq n<b_{k+1}}\delta_{\left(T^{rn}\times T^{sn}\right)(x,x)}\right)\to\rho,
$$
that is, $\widetilde{\rho}$ is a lift of $\rho$. But $(x,x)$ is a generic point for an ergodic measure (formally, we should have considered $\widehat{T}$ and $\widehat{\va}$), see Proposition~\ref{p:d11}), that is, $\rho\in J^e(T^r,T^s)$, so if we know that it has a unique extension to a joining of $\left(\tfs\right)^{r}$ and $\left(\tfs\right)^s$, the joining must be the relatively independent extension of $\rho$ (in other words, up to a permutation of coordinates, $\widetilde{\rho}=\rho\ot\nu\ot\nu$).

\begin{Lemma} If $\beta\notin\Q\alpha+\Q$ then $(\beta,\beta)$ is generic for an ergodic measure $\rho\in J(T^r,T^s)$, for which:
for an arbitrary uniquely ergodic $\cs$ acting on $\ycn$ the only $\widetilde{\rho}\in J(\tfs^r,\tfs^s)$, $\widetilde{\rho}|_{X\times X}=\rho$ is the measure $\rho\ot\nu\ot\nu$.
\end{Lemma} \begin{proof}Recall that the ergodic components of $T^r\times T^s$ are of the form $I_c:=\{(x,y+c) : sx=ry\}$. If $(\beta,\beta)\in I_c$ then
$s\beta=r(\beta-c)$   and $c=\frac{r-s}r\beta$, and therefore $c\notin\Q\alpha+\Q$. The claim follows from Proposition~\ref{p:erg+reg} and Corollary~\ref{c:cherg1}.\end{proof}

It follows that we can now prove the counterparts of Corollaries~\ref{c:affine1}-\ref{c:affine3} for the sequence
$$
[\va^{(n)}(\beta)]=\left[n\beta+\frac{n(n-1)}2\alpha-\frac n2-\sum_{j=0}^{n-1}[\beta+j\alpha ]\right],\;n\geq1,$$
but with the absolute values replaced by parentheses. We have however, the following result.

\begin{Cor}\label{c:uonshort11}
Let $\bfu\colon\N\to\C$ be a multiplicative function satisfying $|\bfu|\leq 1$ and $\ell\geq 2$. Then
$$
\frac1M\sum_{M\leq m<2M}\left|\frac1H\sum_{m\leq h<m+H}e^{2\pi i[\va^{(h)}(\beta)] /\ell}\bfu(h)\right|\to0
$$
when $H\to\infty$, $H/M\to 0$. \end{Cor}
\begin{proof} Consider $Ry:=y+1$ on $\Z/\ell\Z$ and $\chi(y):=e^{2\pi iy/\ell}$. Then, by the above, for each $(b_k)\subset\N$, $b_{k+1}-b_k\to\infty$, and each choice of $(y_k)\subset \Z/\ell\Z$, we have
$$
\frac1{b_{k+1}}\sum_{k\leq K}\left(\sum_{b_k\leq n<b_{k+1}}\chi(R^{[\va^{(n)}(\beta)]}y_k)\bfu(n)\right)\to 0$$
when $K\to\infty$, that is,
$$
\frac1{b_{k+1}}\sum_{k\leq K}\chi(y_k)\left(\sum_{b_k\leq n<b_{k+1}}e^{2\pi i[\va^{(n)}(\beta)]/\ell}\bfu(n)\right)\to 0$$
when $K\to\infty$. If $\ell\geq3$, by choosing $y_k$ so that $\chi(y_k)\left(\sum_{b_k\leq n<b_{k+1}}e^{2\pi i[\va^{(n)}(\beta)]/\ell}\bfu(n)\right)$ belongs to the cone
$\{0\}\cup\{z\in\C :  {\rm arg}(z)\in [-\pi/3,\pi/3]\}$ (or, when $\ell=2$, obtaining $(\chi(y_k))$ and arbitrary sequence of $\pm1$), we obtain
$$
\frac1{b_{k+1}}\sum_{k\leq K}\left|\sum_{b_k\leq n<b_{k+1}}e^{2\pi i[\va^{(n)}(\beta)]/\ell}\bfu(n)\right|\to 0$$
when $K\to\infty$ (cf.\ \cite{Ab-Ku-Le-Ru}). Since $(b_k)$ was arbitrary, as in \cite{Ab-Le-Ru}, we obtain our assertion. 
\end{proof}

\small
\bibliographystyle{siam}
\bibliography{randomSarnakBib.bib}

\bigskip
\footnotesize

\noindent
Joanna Ku\l aga-Przymus\\
\textsc{Institute of Mathematics, Polish Acadamy of Sciences, \'{S}niadeckich 8, 00-956 Warszawa, Poland}\\
\textsc{Faculty of Mathematics and Computer Science, Nicolaus Copernicus University, Chopina 12/18, 87-100 Toru\'{n}, Poland}\par\nopagebreak
\noindent
\href{mailto:joanna.kulaga@gmail.com}
{\texttt{joanna.kulaga@gmail.com}}

\medskip

\noindent
Mariusz Lema\'nczyk\\
\textsc{Faculty of Mathematics and Computer Science, Nicolaus Copernicus University, Chopina 12/18, 87-100 Toru\'{n}, Poland}\par\nopagebreak
\noindent
\href{mailto:mlem@mat.umk.pl}
{\texttt{mlem@mat.umk.pl}}

\end{document}